\DeclareFontFamily{U}{wncy}{}
\DeclareFontShape{U}{wncy}{m}{n}{%
   <5>wncyr5%
   <6>wncyr6%
   <7>wnyr7%
   <8>wncyr8%
   <9>wncyr9%
   <10>wncyr10%
   <11>wncyr10%
   <12>wncyr6%
   <14>wncyr7%
   <17>wncyr8%
   <20>wncyr10%
   <25>wncyr10}{}

\documentclass[a4paper,11pt]{article}
\usepackage{graphicx}
\usepackage{latexsym,color}
\usepackage{amsmath,amsfonts,amssymb,amsthm}
\newtheorem{thm}{Theorem}[section]
\newtheorem{lem}[thm]{Lemma}
\newtheorem{exple}[thm]{Example}

\newtheorem{cor}[thm]{Corollary}
\newtheorem{prop}[thm]{Proposition}

\newtheorem{rem}[thm]{Remark}

\input epsf

\title{On the number of perfect lattices}
\author{Roland Bacher\footnote{This work has been partially supported by the LabEx PERSYVAL-Lab (ANR--11-LABX-0025). The author is a member of the project-team GALOIS supported by this LabEx.}}

\begin{document}
\maketitle

{\sl Abstract\footnote{
Math. class:  Primary: 11H55, Secondary: 11T06, 20K01, 05B30, 05E30.}: 
We show that the number $p_d$ of non-similar 
perfect $d$-dimensional 
lattices satisfies eventually the inequalities
$e^{d^{1-\epsilon}}<p_d<e^{d^{3+\epsilon}}$ for arbitrary small
strictly positive $\epsilon$.}
\vskip.5cm


\section{Main result} 
We denote by $\Lambda_{\min}$ the set of vectors having minimal non-zero 
length in an Euclidean lattice (discrete subgroup of an Euclidean
vector space). An Euclidean
lattice $\Lambda$ of rank 
$d=\dim\left(\Lambda\otimes_{\mathbb Z}\mathbb R\right)$ is \emph{perfect}
if the set $\{v\otimes v\}_{v\in\Lambda_{\min}}$
spans the full ${d+1\choose 2}$-dimensional vector space 
of all symmetric elements in $\left(\Lambda\otimes_{\mathbb Z}\mathbb R\right)\otimes_{\mathbb R} \left(\Lambda\otimes_{\mathbb Z}\mathbb R\right)$.
In the sequel, a lattice will always denote an Euclidean lattice of 
finite rank (henceforth called the dimension of the lattice). 
Every perfect lattice is similar
to an integral lattice and the number of similarity classes of 
perfect lattices of given dimension is finite, cf. for example 
\cite{M}. 
Similarity classes of perfect lattices are in one-to-one correspondence
with isomorphism classes of primitive integral perfect lattices. (A 
lattice is primitive integral if the set of all possible
scalar-products is a coprime set of integers.)
For general information on lattices, the reader can consult \cite{CS}.

The main result of this paper can be resumed as follows:

\begin{thm}\label{thmmain} For every strictly positive $\epsilon$,
the number $p_d$ of 
isomorphism classes of perfect $d$-dimensional primitive integral lattices 
satisfies eventually the inequalities 
$$e^{d^{1-\epsilon}}<p_d<e^{d^{3+\epsilon}}\ .$$
\end{thm}

Otherwise stated, there exist a largest real number $\alpha\in[1,3]$
and a smallest real number $\beta\in[\alpha,3]$ such that 
we have eventually $e^{d^{\alpha-\epsilon}}<p_d<e^{d^{\beta+\epsilon}}$.
I suspect that $1<\alpha=\beta\leq 2$. The inequality $1<\alpha$ is
suggested by the large number of known perfect forms in dimension $8$ and 
$9$ (where we lack a complete classification). I present a 
few non-rigorous arguments for $\beta\leq 2$ in Section 
\ref{sectionheuristics}.

The two inequalities of Theorem \ref{thmmain} are proved by completely 
different methods which give actually explicit lower and upper bounds for 
the numbers $p_d$. 

Lower bounds are obtained by describing an explicit
family of non-isomorphic primitive integral
perfect lattices of minimum $4$. (The minimum of a lattice
is the squared Euclidean norm of a shortest non-zero element.)
Proving perfection of the family is easy. Showing 
that it consists only of non-isomorphic elements is somewhat tedious. Our proof
is based on the fact that ``error correction'' is possible
for a symmetric and reflexive relation obtained by adding a few ``errors'' to
an equivalence relation on a finite  set. Finally, we compute the number of
lattices of dimension $d$ in the family and show that this number grows 
exponentially fast with $d$.
This first part is essentially a refined sequel of \cite{B}. Similar
methods and constructions have been used in \cite{Boe1} and \cite{Boe2}.

Constructions used in this paper and in \cite{B} yield scores of perfect 
integral lattices with minimum $4$. A complete understanding or classification
of such lattices is probably a task doomed to failure. 
Integral perfect lattices of minimum
$2$ are root lattices of type $A,D$ or $E$. They are thus rare and very 
well understood.
Perfect integral lattices of minimum $3$ sit between these two worlds. 
Is there some hope for a (at least partial) classification 
or are there already too many of them?

The upper bound for $p_d$ is also explicit, see Theorem \ref{thmbound}.
We prove it by elementary geometric and combinatorial arguments. Somewhat
weaker bounds (amounting to the eventual inequality
$p_d<e^{d^{4+\epsilon}}$) were obtained by C. Soul\'e in Section 1 of \cite{S}.

A slight modification of the proof of Theorem \ref{thmbound}
gives an upper bound for the number of $\mathrm{GL}_d(\mathbb Z)$-orbits of
$d$-dimensional symmetric lattice
polytopes containing no non-zero lattice points in their interior, 
see Section \ref{sectlattpol}.

We have tried to make this paper as elementary and self-contained as possible.
We apologise for the resulting redundancies with the existing literature.

\section{Perfection}

A subset $\mathcal P$ of a vector space $V$ 
of finite dimension $d$ over
a field $\mathbb K$ of characteristic $\not=2$ is \emph{perfect}
if the set $\{v\otimes v\}_{v\in\mathcal P}$ spans the full 
${d+1\choose 2}$-dimensional vector space 
$\sum_{v,w\in V}v\otimes w+w\otimes v$ of all
symmetric tensor-products in $V\otimes_{\mathbb K} V$. 
Perfection belongs to the realm of linear algebra:
every perfect set $\mathcal P$ contains a perfect 
subset $\mathcal S$ of ${d+1\choose 2}$
elements giving rise to a basis $\{v\otimes v\}_{v\in\mathcal S}$
of symmetric tensor products.
Moreover, every subset of $d$ linearly independent elements in 
a perfect set $\mathcal P$
can be extended to a perfect subset of ${d+1\choose 2}$ elements in $
\mathcal P$.

An example of a perfect set is given by 
$b_1,\dots,b_d,b_i+b_j,1\leq i<j\leq d$ where $b_1,\dots,b_d$ is
a basis of $V$.

Other interesting examples over subfields of real numbers
are given by $4$-designs. (We recall that a $t$-design
is a finite subset $\mathcal S$ of the Euclidean $1-$sphere $\mathbb S^{d-1}$
such that the mean value over $\mathcal S$ of any polynomial of degree at most
$t$ equals the mean value of the polynomial over the unit sphere).
See \cite{Ne} for a survey of B. Venkov's work who discovered
relations between $4$-designs and perfect lattices.

Given a perfect set $\mathcal P$ with ${d+1\choose 2}$ elements in
a $d$-dimensional vector space over a field $\mathbb K$ 
of characteristic $\not=2$, any map $\nu:\mathcal P\longrightarrow \mathbb K$ 
corresponds to a unique bilinear
product $\langle\quad  ,\quad \rangle$ on
$V\times V$ such that 
$\nu(v)=\langle v,v\rangle$ for all $v\in \mathcal P$.
Such a bilinear product is in general not positively defined over a 
real field. Equivalently, a homogeneous quadratic form $q:
V\longmapsto \mathbb K$ is completely defined by its restriction 
to a perfect set.

Perfect lattices are lattices whose set of 
minimal vectors (shortest non-zero elements)
determines the Euclidean structure completely, 
up to similarity. Examples of perfect lattices are given e.g. by
lattices whose minimal vectors form a $4$-design. 
This shows perfection of root lattices of type $A,D,E$, of 
the Leech lattice and of many other interesting lattices with large
automorphism groups.

A useful tool for proving perfection of a set is the following easy Lemma
(see Proposition 3.5.5 in \cite{M} or Proposition 1.1 in \cite{B} for a proof): 

\begin{lem}\label{lemfondperf} Suppose that a set $\mathcal S$ of a vector 
space $V$ intersects a hyperplane $\mathcal H$ (linear subspace of 
codimension $1$) in a perfect subset $\mathcal S'=S\cap \mathcal H$ of
$\mathcal H$ and that $\mathcal S\setminus\mathcal S'$ spans $V$.
Then $\mathcal S$ is perfect in $V$.
\end{lem}

The condition on the span of $\mathcal S\setminus\mathcal S'$
in Lemma \ref{lemfondperf} is necessary as shown by the 
following easy result whose proof is left to the reader:

\begin{lem}\label{lemfonda} Given a linear hyperplane $\mathcal H$ 
of a vector space $V$, any perfect set of $V$ contains a basis
of $V$ not intersecting $\mathcal H$.
\end{lem}

\begin{prop}\label{propfond} Let $\mathcal S$ be a subset of a vector 
space $V$ intersecting two distinct hyperplanes $\mathcal H'$, 
$\mathcal H''$ in two perfect subsets 
$\mathcal S'=\mathcal S\cap \mathcal H'$, $\mathcal S''=\mathcal S\cap \mathcal H''$, of $\mathcal H'$, respectively $\mathcal H''$. Suppose moreover 
that $\mathcal S\setminus(\mathcal S'\cup \mathcal S'')$ 
is non-empty. Then $\mathcal S$ is perfect in $V$.
\end{prop}

Proposition \ref{propfond} implies
immediately the following result for lattices:

\begin{cor}\label{corperfcrit} Let $\Lambda$ be a lattice of 
minimum $m$ and rank $d$ containing two distinct perfect
sublattices $\Lambda'$ and $\Lambda''$ which are both 
of minimum $m$ and rank $d-1$.
Suppose moreover that $\Lambda'+\Lambda''$ is of rank $d$ 
(or, equivalently, of finite index in $\Lambda$) and that $\Lambda$ 
contains a minimal vector $v$ such that $\mathbb Q v\cap \Lambda'=
\mathbb Qv\cap\Lambda''=\{0\}$. Then $\Lambda$ is perfect.
\end{cor}

\begin{proof}[Proof of Proposition \ref{propfond}]
Let $d+1$ be the dimension of $V$.
Lemma \ref{lemfonda} and perfection of $\mathcal S''$ in $\mathcal H''$ 
imply that $\mathcal S''\setminus \mathcal S'$ contains 
a basis $b_1,\dots,b_d$ of $\mathcal H''$. Adding an element 
$b_0\in \mathcal S\setminus\{\mathcal S'
\cup \mathcal S''\}$ to such a basis, we get a basis $b_0,\dots,b_d\subset 
\mathcal S\setminus \mathcal S'$. 
Lemma \ref{lemfondperf} implies perfection of $\mathcal S$ 
in $V$. 
\end{proof}

\section{The lattices $L_d(h_1,h_2,\dots)$: Construction and Results}

Given a strictly increasing sequence $1\leq h_1<h_2<\dots$ of integers,
we denote by $L_d(h_1,h_2,\dots)$ the lattice of all integral vectors in 
$\mathbb Z^{d+2}$ which are orthogonal to $c=(1,1,1,\dots,1,1)\in \mathbb Z^{d+2}$
and to the vector $h=(1,2,\dots,h_1-1,\widehat{h_1},h_1+1,\dots,\widehat{h_2},\dots)\in\mathbb Z^{d+2}$ with strictly 
increasing coordinates given by the $(d+2)$ smallest
elements of $\{1,2,3,4,\dots\}\setminus\{h_1,h_2,h_3,\dots\}$.
We think of the missing coefficients $h_1,h_2,\dots$ as ``holes'' or
``forbidden indices''. Indeed, the lattice $L_d=L_d(h_1,\dots)$ is also 
the set of all vectors $(x_1,x_2,\dots)$ with finite support and 
zero coordinate-sum $\sum_{i=1,2,\dots}x_i=0$ (defining the enumerably 
infinite-dimensional root lattice $A_\infty$ of type $A$) 
such that $\sum ix_i=0$ and non-zero coordinates have indices 
among the first $d+2$ elements 
of $\{1,2,\dots\}\setminus\{h_1,h_2,\dots\}=\{1,2,\dots,\widehat{h_1},\dots\}$.
Equivalently, $L_d$ is the sublattice of $\mathbb Z^{\{1,2,\dots\}}$
supported by the $d+2$ smallest possible 
indices 
such that $L_d$ is orthogonal
to $c=(1,1,1,1,\dots)$ and $h=(1,2,3,4,\dots)$ (the elements $c$ and $h$ 
belong of course only 
to the ``dual lattice'' $\mathbb Z^{\{1,2,\dots\}}$ of the ``infinite dimensional lattice''
$\mathbb Z^{\infty}$ generated by an enumerable orthogonal basis) and $L_d$ is also orthogonal 
to $b_{h_1},b_{h_2},\dots$ with 
$b_i$ denoting the $i$-th element $(0,0,\dots,0,1,0,\dots)$ of the standard
basis of $\mathbb Z^\infty$.

In the sequel, we will often use the notation $\sum_i \alpha_ib_i$
with indices $i\in\{1,2,\dots\}\setminus\{i_1,i_2,\dots\}$ 
corresponding to coefficients 
of $h=(1,2,\dots,h_1-1,\widehat{h_1},h_1+1,\dots)\in\mathbb Z^{d+2}$
when working with elements of $L_d(h_1,\dots)$.

The lattice $L_d(h_1,h_2,\dots)$ is even integral of dimension $d$ with no
elements of (squared euclidean) norm $2$. It is the sublattice of $\mathbb Z^{d+2}$
orthogonal to $\mathbb Zc+\mathbb Zh$ which is a full two-dimensional 
sublattice of $\mathbb Z^{d+2}$ except if $\{1,2,\dots\}\setminus
\{h_1,h_2,\dots\}$ is an arithmetic progression (a case
which will henceforth always be excluded). 
The squared volume $\mathrm{vol}((L_d(h_1,\dots)\otimes_{\mathbb Z}\mathbb R)/
L_d(h_1,\dots)^2)$ of a fundamental domain, 
also called the determinant of $L_d(h_1,h_2,\dots)$, 
equals thus 
$\langle c,c\rangle\langle h,h\rangle-\langle c,h\rangle^2$.

\begin{thm}\label{thmperf} 
$L_d(h_1,h_2,\dots)$ is perfect of minimum $4$ if 
$d\geq 10$ and $h_{i+1}-h_i\geq 6$ for all $i$.
\end{thm}

Given a lattice $L_d(h_1,h_2,\dots)\subset \mathbb Z^{d+2}$ orthogonal 
to $\mathbb Zc+\mathbb Zh$ 
(with $c=(1,1,\dots,1)$ and 
$h=(1,2,3,\dots,\widehat{h_1},\dots,\omega-2,\omega-1,\omega)$) 
we get an isometric lattice of the same form by
considering $L_d(h'_1,h'_2)$ where $\{h'_1,h'_2,\dots\}=
\{\omega+1-h_1,\omega+1-h_2,\dots\}\cap\{1,2,3,\dots\}$.
Indeed, we have $h'=(\omega+1)(1,1,1,\dots,1)-h=(1,2,\dots,\widehat{h'_1},\dots)$,
up to a permutation of coordinates.
We call two such lattices \emph{essentially isomorphic}.
Essentially isomorphic lattices are related
by a suitable affine reflection of their holes. An example 
of two essentially isomorphic lattices is given by
$$L_6(2,5,6,9)=\mathbb Z^8\cap \big(\mathbb Z(1,1,1,1,1,1,1,1)+
\mathbb Z(1,3,4,7,8,10,11,12)\big)^\perp$$ and 
$$L_6(4,7,8,11)=
\mathbb Z^8\cap \big(\mathbb Z(1,1,1,1,1,1,1,1)+
\mathbb Z(1,2,3,5,6,9,10,12)\big)^\perp.$$

The following result yields a large family of non-isomorphic lattices:

\begin{thm}\label{thmautom} If two lattices 
$L_d(h_1,\dots,h_k=d+k+1)$ and $L_d(h'_1,\dots,h'_{k'}=d+k'+1)$ 
of the same dimension $d\geq 46$ satisfy the conditions $h_1,h'_1\geq 7, 
\ h_k=d+k+1,\ h'_{k'}=d+k'+1$ and 
$h_{i+1}-h_i,h'_{i+1}-h'_i\geq 4$ for all $i$, then they are isomorphic
if and only if $k'=k$ and $h'_i=h_i$ for all $i$.
\end{thm}

\section{Proof of Theorem \ref{thmperf}}

The proof for perfection is an induction.
The induction step is the following
special case of Corollary \ref{corperfcrit}:

\begin{prop}\label{propinductionstep} Let $h_1,h_2,\dots\subset\{2,3,4,\dots\}$ be a strictly increasing
sequence of natural integers $\geq 2$ such that 
$L_d(h_1,h_2,\dots)$ and $L_d(h_1-1,h_2-1,h_3-1,\dots)$ are perfect lattices
of minimum $4$ and such that $L_{d+1}(h_1,h_2,\dots)$ contains a minimal
vector $(1,x_2,\dots,x_{d+1},1)$ starting and ending with a coefficient $1$.
Then $L_{d+1}(h_1,h_2,\dots)$ is perfect.
\end{prop}

\proof[Proof of Theorem \ref{thmperf}]
We assume perfection of every $d$-dimensional lattice
$L_d(h_1,h_2,\dots)$ with $h_{i+1}-h_i\geq 6$ for all $i$.

We consider a $(d+1)$-dimensional lattice 
$L_{d+1}(h_1,h_2,\dots)$ with $h_1>1$
(otherwise we remove $h_1$ and shift all holes $h_2,\dots$ by $1$, i.e. we
 consider $L_{d+1}(h_2-1,h_3-1,\dots)$).

By induction, the two $d$-dimensional sublattices
$L_d(h_1,h_2,\dots)$ (consisting of all elements 
in $L_{d+1}(h_1,h_2,\dots)$ with last coordinate $0$)
and $L_d(h_1-1,h_2-1,\dots)$
(consisting of all elements 
in $L_{d+1}(h_1,h_2,\dots)$ with first coordinate $0$)
are perfect. In order to apply Proposition \ref{propinductionstep}, we have
only to show that $L_{d+1}(h_1,\dots)$ contains a minimal 
vector of the form $(1,\dots,1)$, i.e. starting and ending with
a coordinate $1$.
We denote by $h=(1,a,b,c,\dots,w,x,y,z)$ the $(d+3)$-dimensional
vector $(1,2,\dots,h_1-1,\widehat{h_1},h_1+1,\dots,)$.
The condition $h_{i+1}-h_i\geq 6\geq 3$ ensures that both sets 
$\{2,3,4\}$ and $\{z-1,z-2,z-3\}$ contain at most one element in 
the set $\{h_1,\dots\}$ of holes.
There exist thus $s\in \{a,b,c\}$ and $t\in \{w,x,y\}$ such that $s+t=1+z$
ensuring existence of a minimal vector $b_1-b_s-b_t+b_z$ of the form 
$(1,\dots,1)$ in $L_{d+1}(h_1,h_2,\dots)$. 
Proposition \ref{propinductionstep} implies now perfection of 
$L_{d+1}(h_1,\dots)$. 

We have yet to check the initial conditions. It turns out that 
Theorem \ref{thmperf} holds almost in dimension $9$. Indeed, all
lattices $L_9(h_1,\dots)$ with $h_{j+1}-h_j\geq 6$ for all $j$
are perfect (as can be checked by a machine computation)
except the lattice $L_9(4,10)$ (given by all elements
of $\mathbb Z^{11}$ orthogonal to $(1,1,1,1,1,1,1,1,1,1,1)$ and 
$(1,2,3,5,6,7,8,9,11,12,13)$). Fortunately, the two essentially isomorphic 
lattices $L_{10}(4,10)=L_{10}(5,11)$ (which are the two only possible
ways to extend $L_9(4,10)$ into a $10$ dimensional lattice 
of the form $L_{10}(h_1,\dots)$ with $h_{i+1}-h_i\geq 6$)
are perfect. All other $10$-dimensional lattices $L_{10}(h_1,h_2,\dots)$ 
satisfying the conditions of Theorem \ref{thmperf}
are associated to two lattices $L_9(h_1,h_2,\dots)$ and $L_9(h_1-1,h_2-1,\dots)$
which are both non-isomorphic to $L_9(4,10)$.
They are thus perfect by Proposition \ref{propinductionstep}.
\endproof

We display below the list of all $9$-dimensional lattices
$L_9(h_1,\dots)$ with $h_{i+1}-h_i\geq 6$, up to essential 
isomorphism. The last entry is devoted to the $10$-dimensional
lattice $L_{10}(4,10)$. Columns have hopefully understandable meanings
(the last column, labelled $d_2$, gives the rank of the vector space
of symmetric tensors spanned by all vectors $v\otimes v$ for $v$ belonging
to the set $\Lambda_{\min}$ of minimal vectors):
$$\begin{array}{c|c|c|c|c|c}
\hbox{dim}&h_j&\hbox{coordinates of }h&\det&\vert\Lambda_{\min}\vert/2&d_2\\
\hline
9&&1,2,3,4,5,6,7,8,9,10,11&1210&70&45\\
9&2&1,3,4,5,6,7,8,9,10,11,12&1330&66&45\\
9&3&1,2,4,5,6,7,8,9,10,11,12&1426&63&45\\
9&4&1,2,3,5,6,7,8,9,10,11,12&1498&61&45\\
9&5&1,2,3,4,6,7,8,9,10,11,12&1546&60&45\\
9&6&1,2,3,4,5,7,8,9,10,11,12&1570&60&45\\
9&2,8&1,3,4,5,6,7,9,10,11,12,13&1700&56&45\\
9&2,9&1,3,4,5,6,7,8,10,11,12,13&1674&57&45\\
9&2,10&1,3,4,5,6,7,8,9,11,12,13&1624&57&45\\
9&2,11&1,3,4,5,6,7,8,9,10,12,13&1550&60&45\\
9&2,12&1,3,4,5,6,7,8,9,10,11,13&1452&62&45\\
9&3,9&1,2,4,5,6,7,8,10,11,12,13&1778&55&45\\
9&3,10&1,2,4,5,6,7,8,9,11,12,13&1726&56&45\\
9&3,11&1,2,4,5,6,7,8,9,10,12,13&1650&58&45\\
9&4,10&1,2,3,5,6,7,8,9,11,12,13&1804&54&{\bf 44}\\
10&4,10&1,2,3,5,6,7,8,9,11,12,13,14&2507&75&55
\end{array}$$
\endproof

\section{Isomorphic lattices}

The proof of Theorem  \ref{thmautom} is based on the fact that 
metric properties of minimal vectors in a suitable 
lattice $L_d(h_1,h_2,\dots)$ determine the sequence $h_1,h_2,\dots$
up to the essential isomorphism. 

Our main tool for proving this assertion
is a graph-theoretical result of independent interest
described in the next Section.
It gives lower bounds
on the amount of ``tampering'' which does not destroy large equivalence
classes of an equivalence relation on a finite set.

\subsection{$\alpha$-quasi-equivalence classes}\label{subsectalphaequiv}

Consider an equivalence
relation on some finite set which has been slightly 
``tampered with'' and transformed into
a symmetric and reflexive relation which is in general no longer transitive. 
This section describes sufficient
(but not necessarily optimal) conditions on the amount of tampering
which allow the recovery of suitable equivalence classes. 

Symmetric and reflexive relations on a set $V$ are in one-to-one
correspondence with simple graphs with $V$ as their set of vertices. (Recall that 
a simple graph has only undirected edges without multiplicities joining
distinct vertices.)
Given a symmetric and reflexive relation $R$, 
two distinct elements $u,v$ of $V$ are adjacent (joined by an undirected edge) 
if and only if 
$u$ and $v$ are related by $R$. Equivalence relations correspond to disjoint
unions of complete graphs. We use this graph-theoretical framework 
until the end of this Section.

We denote by $\mathcal N_\Gamma(v)$ the set of neighbours (adjacent vertices)
of a vertex $v$ in a simple graph $\Gamma$ and we denote by 
$A\Delta B=(A\cup B)\setminus (A\cap B)$ the symmetric difference 
of two sets $A,B$.

Given a real positive number $\alpha$ in $[0,1/3)$,
a subset $\mathcal C$ of vertices of a finite simple
graph $\Gamma$ is an \emph{$\alpha$-quasi-equivalence class of $\Gamma$} 
if
\begin{align*}
\vert (\mathcal N_\Gamma(v)\cup\{v\})\Delta\mathcal C\vert
&\leq \alpha\vert \mathcal C\vert
\end{align*}
for every vertex $v$ of $\mathcal C$ and
\begin{align*}
\vert\mathcal N_\Gamma(v)\cap\mathcal C\vert&<(1-3\alpha)\vert
\mathcal C\vert
\end{align*}
for every vertex $v$ which is not in $\mathcal C$.

\begin{exple}\label{explequasieqcl} Let $\mathcal C$ be a set of at least $29$ vertices
of a simple graph $\Gamma$. Suppose that $\vert(\mathcal N_\Gamma(v)
\cup \{v\})\Delta\mathcal C\vert\leq 8$ for $v\in\mathcal C$ and
$\vert\mathcal N_\Gamma(v)\cap \mathcal C\vert\leq 4$ for 
$v\not\in \mathcal C$. Then $\mathcal C$ is 
a $\frac{2}{7}$-quasi-equivalence class.
\end{exple}

$0$-quasi-equivalence classes in $\Gamma$ are (vertex-sets of) 
maximal complete subgraphs (also called maximal cliques) of $\Gamma$.
Given a small strictly positive real $\epsilon$, a large 
$\epsilon$-quasi-equivalence class $\mathcal C$ induces almost a
maximal complete 
subgraph: only very few edges (a proportion of at most $\epsilon$)
between a fixed vertex $v\in\mathcal C$ and the remaining vertices of 
$\mathcal C$ can be missing and $v$ can only be adjacent to at most 
$\lfloor \epsilon \vert\mathcal C\vert\rfloor$ 
vertices outside $\mathcal C$. Notice however that a vertex $w\not
\in \mathcal C$ can be adjacent to a very large proportion 
(strictly smaller than 
$(1-3\epsilon)$)
of vertices in $\mathcal C$.

On the other hand, given $\alpha=\frac{1}{3}-\epsilon$ (for small $\epsilon>0$), 
an $\alpha$-quasi-equivalence class $\mathcal C$
can have many missing edges between elements of $\mathcal C$ and it can
have a rather large amount of edges joining 
an element of $\mathcal C$ with elements in the complement of $\mathcal C$.
An element $w\not\in \mathcal C$ is however connected only to a very 
small proportion (of at most $3\epsilon$) of vertices in $\mathcal C$.
Such a class corresponds to a secret organisation 
whose existence is difficult to discover for non-members. 
It displays also some aspects of a connected  component.

\begin{prop}\label{propdistalpha} Distinct $\alpha$-quasi-equivalence classes of 
a finite graph are disjoint.
\end{prop}

\begin{proof}
A common vertex $v$ of two intersecting $\alpha$-quasi-equivalence 
classes $\mathcal C_1$ and $\mathcal C_2$ with 
$\vert\mathcal C_1\vert\geq \vert\mathcal C_2\vert$ gives rise to  
the inequalities
\begin{align*}
\alpha\vert \mathcal C_1\vert&\geq \vert(\mathcal N_\Gamma(v)\cup\{v\})
\Delta\mathcal C_1\vert\\
&\geq \vert(\mathcal C_1\setminus\mathcal C_2)\setminus\mathcal N_\Gamma(v)\vert
\end{align*}
and
\begin{align*}
\alpha\vert \mathcal C_1\vert&\geq\alpha\vert \mathcal C_2\vert\\
&\geq\vert(\mathcal N_\Gamma(v)\cup\{v\})
\Delta\mathcal C_2\vert\\
&\geq \vert(\mathcal C_1\setminus\mathcal C_2)\cap\mathcal N_\Gamma(v)\vert.
\end{align*}
Thus we get
\begin{align*}
2\alpha\vert \mathcal C_1\vert&\geq
\vert\mathcal C_1\setminus\mathcal C_2\vert.
\end{align*}
The trivial inequality
$$\vert C_1\setminus C_2\vert\leq \vert C_1\vert-\vert C_1\cap C_2\vert$$
equivalent to 
$$\vert C_1\cap C_2\vert\leq \vert C_1\vert-\vert C_1\setminus C_2\vert$$
implies now
\begin{align}\label{ineqlemwelldfd}
\vert\mathcal C_1\cap \mathcal C_2\vert\geq
\vert C_1\vert -2\alpha\vert C_1\vert=(1-2\alpha)\vert\mathcal C_1\vert.
\end{align}

Assuming $\mathcal C_1\not=\mathcal C_2$ (and $\vert\mathcal C_1\vert\geq
\vert\mathcal C_2\vert$) we can choose an element $w\in\mathcal C_1\setminus
\mathcal C_2$. We have
\begin{align*}
\alpha\vert \mathcal C_1\vert&\geq \vert(\mathcal N_\Gamma(w)\cup\{w\})
\Delta\mathcal C_1\vert\\
&\geq \vert(\mathcal C_1\cap\mathcal C_2)\setminus\mathcal N_\Gamma(w)\vert
\end{align*}
and
\begin{align*}
(1-3\alpha)\vert \mathcal C_1\vert&\geq(1-3\alpha)\vert \mathcal C_2\vert\\
&>\vert \mathcal N_\Gamma(w)\cap \mathcal C_2\vert\\
&\geq \vert(\mathcal C_1\cap\mathcal C_2)\cap\mathcal N_\Gamma(w)\vert
\end{align*}
which imply
\begin{align*}
(1-2\alpha)\vert \mathcal C_1\vert&>\vert\mathcal C_1\cap \mathcal C_2\vert
\end{align*}
in contradiction with inequality (\ref{ineqlemwelldfd}).
\end{proof}

\subsection{Error-graphs}

The \emph{error-graph} $\mathcal E(\cup_i \mathcal C_i\subset \Gamma)$
of a simple graph $\Gamma$ with a vertex partition $V=\cup_i\mathcal C_i$
into disjoint subsets $\mathcal C_i$ is defined as follows: 
$V=\cup_i \mathcal C_i$ is also the vertex set of 
$\mathcal E=\mathcal E(\cup_i \mathcal C_i\subset \Gamma)$. Two vertices 
$u_{\mathcal E},v_{\mathcal E}$ of $\mathcal E$ are adjacent in $\mathcal E$ if there exists
either an index $i$ such that the corresponding vertices $u_\Gamma$ and $v_\Gamma$ are non-adjacent (in $\Gamma$) and belong to a common subset 
$\mathcal C_i$, or
$u_\Gamma$ and $v_\Gamma$ are adjacent in $\Gamma$
and belong to two different subsets $\mathcal C_i,\mathcal C_j$.

Otherwise stated, $\mathcal E$ is obtained from $\Gamma$ by exchanging
adjacency and non-adjacency in every induced subgraph with vertices 
$\mathcal C_i$.

Edges of the error-graph $\mathcal E(\cup_i \mathcal C_i\subset \Gamma)$
are thus ``errors'' of the symmetric relation on $V$ 
encoded by (edges of) $\Gamma$
with respect to the equivalence relation with equivalence 
classes $\mathcal C_i$.

\subsection{Neighbourhoods}\label{sectionneighbourhoods}

We consider the set $\Lambda_{\min}$ of all minimal vectors
in a fixed lattice $\Lambda=L_d(h_1,h_2,\dots)$ with minimum $4$.
Two minimal vectors $v,w\in \Lambda_{\min}$ are \emph{neighbours}
if $\langle v,w\rangle=2$. 
The set $\mathcal N(v)$ of neighbours of a given element 
$v\in
\Lambda_{\min}$ can be partitioned into six disjoint subsets
$$\mathcal N(v)=\mathcal F_{**00}(v)\cup \mathcal F_{*0*0}(v)\cup
\mathcal F_{*00*}(v)\cup\mathcal F_{0**0}(v)\cup\mathcal F_{0*0*}(v)
\cup\mathcal F_{00**}(v)$$
with stars, respectively zeros, indicating coordinates $\pm1$, 
respectively $0$, in the support $\{i,i+k,j-k,j\},\ i<i+k<j-k<j$ of 
$v=b_i-b_{i+k}-b_{j-k}+b_k$.

The involution $\iota:w\longrightarrow v-w$ induces a one-to-one correspondence
between the two subsets of the three pairs 
$$\{\mathcal F_{**00}(v),\mathcal F_{00**}(v)\},\{\mathcal F_{*0*0}(v),\mathcal F_{0*0*}(v)\},\{\mathcal F_{*00*}(v),\mathcal F_{0**0}(v)\}\ .$$
We call such pairs \emph{complementary} and denote them using the hopefully 
self-explanatory notations
\begin{align*}
\mathcal F_{aabb}(v)&=\mathcal F_{**00}(v)\cup\mathcal F_{00**}(v),\\
\mathcal F_{abab}(v)&=\mathcal F_{*0*0}(v)\cup\mathcal F_{0*0*}(v),\\
\mathcal F_{abba}(v)&=\mathcal F_{*00*}(v)\cup\mathcal F_{0**0}(v).\\
\end{align*}
We write 
$\overline{\mathcal F_{aabb}(v)},\overline{\mathcal F_{abab}(v)},\overline{\mathcal F_{abba}(v)}$ for the orbits under $\iota$ of the three sets 
$\mathcal F_{aabb}(v),\mathcal F_{abab}(v),\mathcal F_{abba}(v)$.
Either $\mathcal F_{**00}(v)$ or $\mathcal F_{00**}$ represent all elements 
of $\overline{\mathcal F_{aabb}(v)}$. The same
statement holds for 
$\overline{\mathcal F_{abab}(v)},\overline{\mathcal F_{abba}(v)}$.
Henceforth, we identify often a vector $w\in\mathcal N(v)$ with 
its class in $\overline{\mathcal N(v)}$.

The following table lists all 6 elements of the set $\mathcal N(
0,1,0,-1,-1,0,1)$ in $L_5(\emptyset)=
\left(\mathbb Z(1,1,1,1,1,1,1)+\mathbb Z(1,2,3,4,5,6,7)\right)^\perp\subset
\mathbb Z^7$
together with the sets $\mathcal F_{**00},\mathcal F_{*0*0},\mathcal F_{*00*},
\mathcal F_{0**0},\mathcal F_{0*0*},\mathcal F_{00**}$ and 
$\mathcal F_{aabb},\mathcal F_{abab},\mathcal F_{abba}$ (with dropped 
common argument $v=(0,1,0,-1,-1,0,1)\in\Lambda_{\min}$) containing them: 
$$\begin{array}{ccccccc|c|c}
0&1&0&-1&-1&0&1&&\\
\hline
-1&1&1&-1&0&0&0&\in\mathcal F_{**00}&\subset\mathcal F_{aabb}\\
1&0&-1&0&-1&0&1&\in\mathcal F_{00**}&\subset\mathcal F_{aabb}\\
0&1&-1&0&-1&1&0&\in\mathcal F_{*0*0}&\subset\mathcal F_{abab}\\
0&0&1&-1&0&-1&1&\in\mathcal F_{0*0*}&\subset\mathcal F_{abab}\\
0&1&-1&0&0&-1&1&\in\mathcal F_{*00*}&\subset\mathcal F_{abba}\\
0&0&1&-1&-1&1&0&\in\mathcal F_{0**0}&\subset\mathcal F_{abba}\\
\end{array}$$

Since 
$\langle u,v-w\rangle=2-\langle u,w\rangle$, the parity
of the scalar product is well-defined on $\overline{\mathcal N(v)}$. 
We get thus a map 
$\overline{\mathcal N(v)}\times \overline{\mathcal N(v)}\longrightarrow 
\mathbb Z/2\mathbb Z$ where $\overline{\mathcal N(v)}=
\overline{\mathcal F_{aabb}(v)}\cup\overline{\mathcal F_{abab}(v)}\cup \overline{\mathcal F_{abba}(v)}$. We say that two classes represented by $u,w$ have a 
generic scalar product if $\langle u,v\rangle\equiv 0\pmod 2$ if and only
if $u,v$ belong both to the same subset 
$\overline{\mathcal F_{aabb}(v)},\overline{\mathcal F_{abab}(v)},\overline{\mathcal F_{abba}(v)}$ of $\overline{\mathcal N(v)}$.
Generic values are given by the table
$$\begin{array}{c||c|c|c}
&\overline{\mathcal F_{aabb}}&\overline{\mathcal F_{abab}}&\overline{
\mathcal F_{abba}}\\
\hline\hline
\overline{\mathcal F_{aabb}}&{\bf 0}&{\bf 1}&{\bf 1}\\
\hline
\overline{\mathcal F_{abab}}&{\bf 1}&{\bf 0}&{\bf 1}\\
\hline
\overline{\mathcal F_{abba}}&{\bf 1}&{\bf 1}&{\bf 0}\\
\end{array}$$
Non-generic values, often called errors in the sequel, occur 
at most $8$ times with a given first element in $\overline{\mathcal N(v)}$.
More precisely, the table 
$$\begin{array}{c||c|c|c}
&\overline{\mathcal F_{aabb}}&\overline{\mathcal F_{abab}}&\overline{
\mathcal F_{abba}}\\
\hline\hline
\overline{\mathcal F_{aabb}}&\leq 2&\leq 4&\leq 2\\
\hline
\overline{\mathcal F_{abab}}&\leq 4&\leq 2&\leq 2\\
\hline
\overline{\mathcal F_{abba}}&\leq 2&\leq 2&0\\
\end{array}$$
displays upper bounds for the number of errors 
(of the map $\overline{\mathcal N(v)}\times \overline{\mathcal N(v)}\longrightarrow 
\mathbb Z/2\mathbb Z$) occurring 
with a fixed element of $\overline{\mathcal N(v)}$.

We illustrate the first line by considering 
$v=(0,0,0,0,1,-1,1,-1,0,0,\dots)\in L_{d\geq 12}(h_1\geq 15,\dots)$. Vectors $e_1,\dots,e_8$
representing all errors within $\overline{\mathcal F(v)}$
occurring with the class of 
$w=(0,0,0,0,1,-1,0,0,0,0,-1,1,0,\dots)$ in $\overline{\mathcal F_{aabb}(v)}$
are given by
$$\begin{array}{c|cccccccccccccc|c}
v  &0&0&0&0&+&-&-&+&0&0&0&0&0&0&\langle w,e_j\rangle\\
\hline\hline
w=e_0  &&&&&+&-&&&&&-&+&&&4\\
\hline\hline
e_1&&&&&+&-&&&&-&+&&&&1\\
e_2&&&&&+&-&&&&&&-&+&&1\\
\hline
e_3&&&&&+&&-&&-&&+&&&&0\\
e_4&&&&&+&&-&&&-&&+&&&2\\
e_5&&&&&+&&-&&&&-&&+&&2\\
e_6&&&&&+&&-&&&&&-&&+&0\\
\hline
e_7&&-&&&+&&&+&&&-&&&&2\\
e_8&-&&&&+&&&+&&&&-&&&0\\
\end{array}$$
where we keep only the signs of non-zero coefficients. The elements
$e_1,e_2$ realise the maximal number of two 
errors occurring
within $\overline{\mathcal F_{aabb}(v)}$, the vectors $e_3,\dots,e_6$ realise
the maximal number of $4$ errors occurring between the class of $w$ in 
$\overline{\mathcal F_{aabb}(v)}$ and $\overline{\mathcal F_{abab}(v)}$
and $e_7,e_8$ realise the maximal number of $2$ errors between the class of $w$ and $\overline{\mathcal F_{abba}(v)}$.

We consider $\overline{\mathcal N(v)}$ as the vertex-set of a graph
with edges given by pairs of different vertices with an even scalar
product among representatives. 

\begin{prop}\label{propclwdefinN} If at least two of the 
three classes $\overline{\mathcal F_{aabb}(v)}, 
\overline{\mathcal F_{abab}(v)},\overline{\mathcal F_{abba}(v)}$ contain 
at least $29$ elements then all three classes 
are uniquely
defined in terms of the graph-structure on $\overline{\mathcal N(v)}$.
\end{prop}

\proof{} Example \ref{explequasieqcl} and the above bounds for the maximal 
number of errors (non-generic values of scalar products) show that two 
such classes with at least $29$ elements define 
$\frac{2}{7}$-quasi-equivalence
relations in $\overline{\mathcal N(v)}$. They are thus well-defined
by Proposition \ref{propdistalpha}. The third class is given by 
the remaining elements.
\endproof

\subsection{The path $\mathcal P$
associated to $(1,-1,-1,1,0,\dots,0)$}

\begin{prop}\label{propvdetholes}
Let $v$ denote the minimal vector $b_1-b_2-b_3+b_4=(1,-1,-1,1,0,0,
\dots,0)$ of a fixed lattice $L_d(h_1,\dots,h_k=d+1+k)$ satisfying the
conditions of Theorem \ref{thmautom}.

The set $h_1,\dots,h_k$ of holes is uniquely determined by 
the graph-structure on the set of equivalence 
classes $\overline{\mathcal N(v)}$ of neighbours of $v$.
\end{prop}

\proof{} We keep the notations of Section \ref{sectionneighbourhoods}
except for dropping the argument $v=(1,-1,-1,1,0,\dots)$ for 
subsets of the sets $\mathcal N=\mathcal N(v)$ or 
$\overline{\mathcal N}=\overline{\mathcal N(v)}$.

The set $\overline{\mathcal F_{abba}}$ is empty and
the conditions $h_1\geq 7, h_k=d+k+1$ and $h_{i+1}-h_i\geq 4$ 
imply that
$\mathcal A=\overline{\mathcal \mathbb F_{aabb}}$ and 
$\mathcal B=\overline{\mathcal F_{abab}}$ have both exactly $d-3-k$ elements.
Indeed, among the $k+d-3$ vectors $b_1-b_2-b_{i-1}+b_i,\ i=6,\dots,k+d+2$,
exactly $2k$ vectors of the form $b_1-b_2-b_{h_i-1}+b_{h_i},b_1-b_2-b_{h_i}+
b_{h_i+1}$ are not orthogonal to all elements $b_{h_1},\dots,b_{h_k}$.
A similar argument works for $\mathcal B$.

We construct an oriented path $\mathcal P$ as follows: 
We start with the error-graph $\mathcal E=
\mathcal E(\mathcal A\cup
\mathcal B\subset \overline{\mathcal N})$ of 
$\overline{\mathcal N}$. We denote by $\mathcal E'$ the subgraph 
of $\mathcal E$ obtained by removing all vertices of 
$\mathcal B=\overline{\mathcal F_{abab}}$ which
are involved in a triangle with the two remaining
vertices in $\mathcal A=\overline{\mathcal \mathbb F_{aabb}}$. 
Such triangles are given
by three vectors $b_1-b_3,b_{l-2}+b_l\in\mathcal B,\ b_1-b_2-b_{l-2}+b_{l-1},\ b_1-b_2-b_{l-1}+b_l\in \mathcal A$ with $\{l-2,l-1,l\}\subset\{5,\dots,d+k\}$ not intersecting $\{h_1,\dots,
h_k\}$.) We denote by $\mathcal H$ the remaining set of 
$k$ vertices of $\mathcal B$. Elements of $\mathcal H$
are of the form $b_1-b_3-b_{h_i-1}+b_{h_i+1}$ and 
correspond to the $k$ holes $h_1,h_2,\dots$. The graph $\mathcal E'$ 
is a path-graph (or segment, or Dynkin graph of type $A$) 
having one leaf (vertex of degree one)
in $\mathcal A$ and one leaf in $\mathcal H\subset \mathcal B$.
We orient its edges in order to get an oriented path 
$$\begin{array}{cc}
b_1-b_2-b_5+b_6,b_1-b_2-b_6+b_7,\dots,b_1-b_2-b_{h_1-2}+b_{h_1-1},\\
b_1-b_3-b_{h_1-1}+b_{h_1+1},b_1-b_2-b_{h_1+1}+b_{h_1+2},\dots,
b_1-b_2-b_{h_2-2}+b_{h_2-1},\\
\vdots\\
b_1-b_3-b_{h_{k-1}-1}+b_{h_{k-1}+1},b_1-b_2-b_{h_{k-1}+1}+b_{h_{k-1}+2},\dots,\\
,\dots,b_1-b_2-b_{d+k-1}+b_{d+k},b_1-b_3-b_{d+k}+b_{d+k+2}
\end{array}$$
starting at the  
vertex $b_1-b_2-b_5+b_6$ of $\mathcal A$
and ending at the vertex $b_1-b_3-b_{d+k}+b_{d+k+2}$ of
$\mathcal H\subset\mathcal B$. 

The final path $\mathcal P$ is obtained from $\mathcal E'$ as follows:
We add $k$ additional vertices to the vertex-set $\mathcal A$
by considering the midpoints 
of the $k$ oriented edges $(b_1-b_2-b_{h_i-2}+b_{h_i-1},b_1-b_3-b_{h_i-1}+b_{h_i+1})$
which start in $\mathcal A$ and end in $\mathcal H$. These oriented 
edges are well-defined since two distinct vertices 
$b_1-b_3-b_{h_i-1}+b_{h_i+1},b_1-b_3-b_{h_j-1}+b_{h_j+1}$
of $\mathcal H$ are never adjacent 
in $\mathcal E'$ (they are always separated by at least two vertices 
$b_1-b_2-b_{h_i+1}+b_{h_i+2},b_1-b_2-b_{h_i+2}+b_{h_i+3}\in \mathcal A$ if $i<j$) and
since the initial vertex of $\mathcal E'$ does not belong to $\mathcal H$.
We label now the $d-3+k=(d-k-3+k)+k$ vertices of $\mathcal P$ increasingly by 
$5,6,\dots,d+k+1$. The set of labels of the $k$
vertices in $\mathcal H$ defines now the sequence $h_1,\dots,h_k=d+k+1$.
\endproof

\subsection{Admissible vectors}

\begin{lem} \label{lemerrorsinclasses} Both sets $\mathcal A=
\overline{\mathcal F_{aabb}(v)}$ and 
$\mathcal B=\overline{\mathcal F_{abab}(v)}$ have errors for 
$v=b_1-b_2-b_3+b_4=(1,-1,-1,1,0,\dots)$ 
in a lattice $L_d(h_1,\dots,h_k=d+1+k)$
satisfying the conditions of Theorem \ref{thmautom}.
\end{lem} 

Recall that an error in $\mathcal A$ (respectively $\mathcal B$) is given 
by two minimal vectors $s,t$ in $\mathcal A$ (respectively in
$\mathcal B$) having an odd scalar 
product $\langle s,t\rangle\equiv 1\pmod 2$. 

\proof[Proof of Lemma \ref{lemerrorsinclasses}]
An error in $\mathcal A$ is realised by
$b_1-b_2-b_{d+k-2}+b_{d+k-1},b_1-b_2-b_{d+k-1}+b_{d+k}$.

An error in $\mathcal B$ is realised by
$b_1-b_3-b_{d+k-2}+b_{d+k},b_1-b_3-b_{d+k}+b_{d+k+2}$.
\endproof

\begin{lem}\label{lemupperbdk} We have the inequality $k\leq 
\lfloor(d-2)/3\rfloor$ for the number $k$ of holes in a $d$-dimensional
lattice $L_d(h_1,\dots,h_k=d+k+1)$ satisfying the conditions of 
Theorem \ref{thmautom}.
\end{lem}

Equality is achieved for $h_i=3+4i$ for $i=1,2,\dots$. 

We leave the easy proof
(based on the pigeon-hole principle) to the reader.

A minimal vector $u$ of a lattice $L_d(h_1,\dots,h_k=d+k+1)$ 
satisfying the conditions of Theorem \ref{thmautom} is called 
\emph{admissible} if the following conditions hold:
\begin{enumerate}
\item{} $\overline{N(u)}$ has an even number of at least
$2\lceil (2d-7)/3\rceil$ elements defining two
$\frac{2}{7}$-quasi-equivalence classes $\mathcal A,\mathcal B$ 
of equal size.
\item{} Both classes $\mathcal A,\mathcal B$ contain errors
(i.e. do not define complete subgraphs of $\overline{\mathcal N(u)}$). 
\item{} The error graph $\mathcal E(\mathcal A\cup\mathcal B\subset 
\overline{\mathcal N(u)})$ contains a triangle with two vertices in 
$\mathcal A$ and one vertex in $\mathcal B$.
\item{} The construction 
of the oriented path $\mathcal P$ explained in the proof
of Proposition \ref{propvdetholes} works and yields an oriented path 
with $2d-6-\vert\overline{\mathcal N(u)}\vert/2$ vertices.
\end{enumerate}

The proof of Theorem \ref{thmautom} follows now from the following result:

\begin{prop}\label{propadmissvect}
A lattice $L_d(h_1,\dots,h_k=d+k+1)$ satisfying the conditions of Theorem \ref{thmautom} 
contains a unique pair $\pm v$ of admissible minimal vectors.
\end{prop}

\proof{}
We show first that 
$v=v_c=b_1-b_2-b_3+b_4=(1,-1,-1,1,0,0,\dots,0)$ is admissible. 
The sets $\mathcal A=\overline{\mathcal F_{aabb}(v)}$ and 
$\mathcal B=\overline{\mathcal F_{abab}(v)}$ contain both 
$d-3-k$ elements. Lemma \ref{lemupperbdk} yields 
$d-3-k\geq \lceil (2d-7)/3\rceil\geq 29$ for $d\geq 46$.
The sets $\mathcal A$ and $\mathcal B$ are thus $\frac{2}{7}$-quasi
equivalence classes by Example \ref{explequasieqcl}. (See also the proof of
Proposition \ref{propclwdefinN}.)

Lemma \ref{lemerrorsinclasses} shows that both classes $\mathcal A$ 
and $\mathcal B$ contain errors.

A triangle in the error graph $\mathcal E(\mathcal A\cup\mathcal B\subset 
\overline{\mathcal N(v)})$ with two vertices in $\mathcal A$ and a 
last vertex in $\mathcal B$ is defined by 
$b_1-b_2-b_{d+k-2}+b_{d+k-1},b_1-b_2-b_{d+k-1}+b_{d+k},b_1-b_3-b_{d+k-2}+b_{d+k}$.

The last condition is fulfilled as shown by the proof of 
Proposition \ref{propvdetholes} and yields the path 
$\mathcal P$ with $d+k-3=2d-6-(d-k-3)$ vertices.

We consider now an admissible minimal vector $u$ of 
$L_d(h_1,\dots,h_k=d+k+1)$. Condition (1) for admissibility
shows that $\overline{\mathcal N(u)}$ consists of two 
$\frac{2}{7}$-quasi-equivalence classes $\mathcal A,\mathcal B$
containing both at least $\lceil (2d-7)/3\rceil\geq \lceil (2\cdot 46-7)/3
\rceil=29$ elements. By Proposition \ref{propclwdefinN}, 
they define thus two of the three classes
$\overline{\mathcal F_{aabb}},\overline{\mathcal F_{abab}},
\overline{\mathcal F_{abba}}$. 

Since $\overline{\mathcal F_{abba}}$
does not contain internal errors, these classes are
$\overline{\mathcal F_{aabb}}$ and $\overline{\mathcal F_{abab}}$.

Condition (3) for admissibility implies that $\mathcal A=\overline{
\mathcal F_{aabb}},\ \mathcal B=\overline{
\mathcal F_{abab}}$ and that $u$ is of 
the form $b_a-b_{a+p}-b_{a+2p}+b_{a+3p}$ for some strictly positive integers 
$a$ and $p$. We can exclude $p\geq 3$: 
Indeed, both sets $\{a+p-2,a+p-1,a+p+1\}$
and $\{a+2p-1,a+2p+1,a+2p+2\}$ intersect
$\{h_1,\dots,h_k=d+k+1\}$ in at most a single element. For $p\geq 3$,
there exists thus $\epsilon\in \{2,1,-1\}$ such that
the lattice $L_d(h_1,\dots)$ contains the minimal vector
$b_a-b_{a+p-\epsilon}-b_{a+2p+\epsilon}+b_{a+3p}$ in contradiction with
$\overline{\mathcal F_{abba}}=\emptyset$. 
In order to exclude $p=2$ we use condition (4): Suppose indeed that $p=2$.
If there exists two elements $w_i=b_a-b_{a+2}-b_{a_i-2}+b_{a_i},i=1,2$ in 
$\overline{\mathcal F_{aabb}}$ with $a_1\not\equiv
a_2\pmod 2$, then the construction of the oriented path
of Proposition \ref{propvdetholes} fails since $w_1$ and $w_2$ 
define vertices in two different connected components of the 
intermediary graph $\mathcal E'$. 
Two such elements $w_1,w_2$ do not necessarily exist if the sequence
of holes consist of at most two arithmetic progressions of step $4$.
In this case, there exist two elements 
$t_i=b_a-b_{a+4}-b_{c_i-4}+b_{c_i},i=1,2$ in 
$\overline{\mathcal F_{abab}}$ with $c_1\not\equiv
c_2\pmod 4$, also leading to a disconnected graph $\mathcal E'$.
(The case $p=2$ can also be excluded by showing that it cannot
lead to a path $\mathcal P$ having the correct length.)

We have thus $p=1$ and $u$ is of the form $u=b_a-b_{a+1}-b_{a+2}+b_{a+3}$.
The intermediary graph $\mathcal E'$ constructed in the proof of 
Proposition \ref{propvdetholes} is connected if an only if 
$a\in \{1,2,d+k-3\}$. The case $a=2$ and $h_1=7$ leads to a graph $\mathcal E'$
with both leaves in $\mathcal H\subset\overline{\mathcal F_{abab}}$
and is thus excluded. The 
case $a=2$ and $h>7$ leads to a sequence of holes
defining the lattice $L_{d-1}(h_1-1,h_2-1,\dots,h_k+d)$ which is only
$(d-1)$-dimensional. The case $a=d+k-3$ leads to a graph $\mathcal E'$
having both leaves in $\mathcal A$ except if $h_{k-1}=d+k-5$.
This last case leads also to a sequence defining a lattice of dimension 
$d-1$ as follows: The corresponding sets $\overline{\mathcal F_{aabb}}$
and $\overline{\mathcal F_{abab}}$ have both $d-3-k$ elements. 
They lead to a path of length $d-3-k+2(k-1)=(d-1)+(k-1)-3$
defining a set $h_1,\dots,h_{k-1}$ of $k-1$ holes for 
a lattice of dimension $d-1$.
\endproof

\subsection{Proof of Theorem \ref{thmautom}}

\proof[Proof of Theorem \ref{thmautom}] 
We construct the set of minimal vectors 
(of a lattice as in Theorem \ref{thmautom})
and use it 
for determining all minimal admissible vectors.
They form a unique pair by Proposition \ref{propadmissvect}.
This pair corresponds thus to the vectors $\pm (1,-1,-1,1,0,\dots)$
which determine the set $h_1,\dots$ of holes uniquely by 
Proposition \ref{propvdetholes}.
\endproof

\section{A linear recursion}\label{sectalpha}

We denote by $\alpha(n)$ the number of strictly increasing 
integer-sequences of length $n$ which 
start with $1$ and which have no missing integers 
at distance strictly smaller than
$6$. A sequence $s_1=1,s_2,\dots,s_n$ contributes thus $1$ to $\alpha_n$
if and only if $s_{i+1}-s_i\in\{1,2\}$ for $i=1,\dots,n-1$ and
$s_{i+1}-s_i=2,s_{j+1}-s_j=2$ for $1\leq i<j\leq n-1$ implies $s_j-s_i\geq 5$.

\begin{prop}\label{propalphan} We have $\alpha(i)=i$ for $i=1,\dots,6$. For $n\geq 6$
we have the recursion
$$\alpha(n)=\alpha(n-1)+\alpha(n-5).$$

In particular, the sequence $\alpha(1),\alpha(2),\dots$ grows 
exponentially fast.
\end{prop}

The first few terms of $\alpha(1),\dots$ are given by
$$1,2,3,4,5,6,8,11,15,20,26,\dots.$$

\proof[Proof of Proposition \ref{propalphan}] 
The sequence $1,2,3,4,\dots,n$ is the unique contribution to 
$\alpha(n)$ without holes 
(missing integers in the set $\{1=s_1,s_2,\dots,s_n\}$) and
there are $n-1$ sequences (contributing $1$ to $\alpha(n)$)
with exactly $1$ hole. Since contributions to $\alpha(n)$ for $n\leq 6$ have
at most one hole, we get the formula $\alpha(i)=i$ for $i=1,\dots,6$.

For $n\geq 6$, a contribution of $1$ to $\alpha(n)$ is either given
by $1,\dots,s_{n-1}=\omega-1,s_n=\omega$ for some integer $\omega\geq n$ and 
such contributions are in one-to-one correspondence with 
contributions to $\alpha(n-1)$ (erase the last term) or it is of the 
form 
$$1,\dots,s_{n-5}=\omega-6,\omega-5,\omega-4,s_{n-2}=\omega-3,s_{n-1}=\omega-2,s_n=\omega$$
with $s_1=1,\dots,s_{n-5}=\omega-6$ an arbitrary contribution to $\alpha(n-5)$.

The sequence $\alpha(1),\dots$ 
satisfies the linear recursion $\alpha(n)=\alpha(n-1)+\alpha(n-5)$
with characteristic polynomial $P=z^5-z^4-1=(z^2-z+1)(z^3-z-1)$.
Straightforward (but tedious) computations show
$$\alpha(n)=\frac{1}{7}\pi(n)+\frac{1}{161}\sum_{\rho,\rho^3-\rho=1}(45\rho^n+
61\rho^{n+1}+36\rho^{n+2})$$
where $\pi(n)=(-1-2/\sqrt{-3})e^{in\pi/3}+(-1+2/\sqrt{-3})e^{-in\pi/3}$
is $6$-periodic given by
$$\begin{array}{c||c|c|c|c|c|c}
n=&6k&6k+1&6k+2&6k+3&6k+4&6k+5\\
\hline
\pi(n)=&-2&-3&-1&2&3&1
\end{array}$$
and where the sum is over the three roots of the polynomial $z^3-z-1$.
The sequence $\alpha_n$ satisfies thus
$\lim_{n\rightarrow\infty}\frac{\alpha(n)}{\rho^n}=\frac{45+
61\rho+36\rho^2}{161}$
where $\rho\sim1.324718$ is the unique real root of $z^3-z-1$.
\endproof

\begin{rem} (i) Setting
$\tilde\alpha(i)=1+(i-1)t$ for $i=1,\dots,6$ and 
$\tilde\alpha(n)=\tilde\alpha(n-1)+t\tilde\alpha(n-5), n\geq 6$
we get a sequence $\tilde\alpha(1),\tilde\alpha(2),\dots$ specialising 
to $\alpha(n)$ at $t=1$ with coefficients counting contributions to 
$\alpha(n)$ according to their number of holes (i.e. the coefficient
of $t^k$ in $\tilde\alpha(n)$ is the number of sequences
contributing to $\alpha(n)$ which are of the form 
$s_1=1,s_2,\dots,s_n=n+k$).

(ii) A trivial exponential 
lower bound on $\alpha(n)$ can be obtained as follows:
Remove either the $6$-th or the $7$-th element in 
every set $\{7l+1,7l+2,\dots,7l+6,7l+7\}$.
This shows $\alpha(n)\geq 2^{\lfloor n/6\rfloor}$.
\end{rem}

\section{An exponentially large family of perfect lattices}

\begin{prop}\label{propauxexp} The number of $d$-dimensional lattices satisfying the 
conditions of Theorem \ref{thmperf}  and of Theorem \ref{thmautom}
equals $\alpha(d-8)$ for $d\geq 46$ where $\alpha(1),\dots$ is the 
sequence of Section \ref{sectalpha} recursively defined
by $\alpha(n)=n$ for $n=1,\dots,5$
and by $\alpha(n)=\alpha(n-1)+\alpha(n-5)$ for $n\geq 6$. 
\end{prop}

\proof{} Such a lattice $L_d(h_1,\dots)$ corresponds to a vector 
$$h=(1,2,3,4,5,s_1+5,s_2+5,\dots,\omega-6=s_{d-8}+5,\omega-5,
\omega-4,\omega-3,\omega-2,\omega)$$
where $s_1=1,s_2,\dots,s_{d-8}=\omega-11$ is a sequence contributing
to $\alpha(d-8)$. This construction is one-to-one for $d\geq 46$:
every sequence $s_1,\dots,s_{d-8}$ contributing to $\alpha(d-8)$
corresponds to a lattice satisfying the conditions of 
Theorem \ref{thmperf}  and of Theorem \ref{thmautom}.
\endproof


\section{Maximal indices for pairs of well-rounded lattices with the same minimum}\label{sectId}

A lattice $\Lambda$ is \emph{well-rounded} if its minimal vectors 
span the ambient
vector space $\Lambda\otimes_{\mathbb Z}\mathbb R$.

We denote by $I_d$ the smallest integer such that every
well-rounded $d$-dimensional sublattice $\Lambda'$ with minimal vectors
contained in the set of minimal vectors of a $d$-dimensional lattice 
$\Lambda$ has index at most $I_d$ in $\Lambda$. The integers $I_d$
(and refinements describing all possible group-structures of
the quotient group $\Lambda/\Lambda'$ in small dimensions) 
have been studied by several authors, see for example \cite{Ke}.

The maximal index $I_d$ is of course realised by a 
sublattice generated by $d$ suitable minimal vectors of a suitable 
well-rounded $d$-dimensional lattice.

Since perfect lattices are well-rounded,
the definition of $I_d$ implies immediately:

\begin{prop}\label{propIdbound}
A sublattice $\sum_{i=1}^d\mathbb Zv_d$ generated by $d$ 
linearly independent minimal elements of a perfect 
$d$-dimensional lattice $\Lambda$ is of index at most $I_d$ 
in $\Lambda$.
\end{prop}

We give now upper bounds for $I_d$.

\subsection{An upper bound for $I_d$ from Minkowski's inequality}


\begin{prop}\label{propuppMink} 
We have
\begin{align}\label{ineqMink}
I_d\leq \left\lfloor\left(\frac{4}{\pi}
\right)^{d/2}\left(\frac{d}{2}\right)!\right\rfloor\ .
\end{align}
\end{prop}

The proof of Proposition \ref{propuppMink} 
uses Minkowski's inequality (see Chapter 3 of \cite{Mi}) 
stating that a centrally symmetric convex subset $\mathcal C=-\mathcal C$
of $\mathbb R^d$ 
has volume at most $2^d$ if it contains no non-zero elements of $\mathbb Z^d$
in its interior.

\proof[Proof of Proposition \ref{propuppMink}]
Without loss of generality, we can consider a sublattice 
$\Lambda'=\sum_{i=1}^d\mathbb Zb_i$ generated by 
$d$ linearly independent minimal vectors $b_1,\dots,b_d$ of norm $1$ 
of a $d$-dimensional well-rounded lattice $\Lambda$ with minimum $1$.
We have the inequality 
$V'\leq 1$ for the volume $V'=\mathrm{vol}(\mathbb R^d/\Lambda')$
of a fundamental domain for $\Lambda'$. (The equality $V'=1$ 
holds if and only if $b_1,\dots,b_d$ is an orthonormal basis.
The lattice $\Lambda'$ is then 
the standard lattice $\mathbb Z^d$.)

Applying Minkowski's inequality 
$\mathrm{vol}(\mathbb B^d)\leq 2^dV$
where $V=\mathrm{vol}(\mathbb R^d/\Lambda)$ to the $d$-dimensional
Euclidean unit ball $\mathbb B^d$ of volume $\frac{\pi^{d/2}}{(d/2)!}$
(see for example \cite{CS}, Formula (17) of Chapter 1)
we get the inequality
$$V'\leq 1=\frac{(d/2)!}{\pi^{d/2}}\mathrm{vol}(\mathbb B^d)\leq \frac{(d/2)!}{\pi^{d/2}}2^dV$$
showing $\frac{V'}{V}\leq \left(\frac{4}{\pi}
\right)^{d/2}\left(\frac{d}{2}\right)!$ for the index $V'/V\in\mathbb N$ 
of the lattice
$\Lambda'$ in $\Lambda$.
\endproof

\subsection{Upper bounds for $I_d$ in terms of Hermite's constants}

I thank Jacques Martinet for drawing my attention to the ideas of 
this section which is essentially
identical with Section 2 of \cite{Ke}. The scope of the following lines
is to provide the lazy reader with an armchair.

We denote by $\min(\Lambda)=\min_{v\in\Lambda\setminus\{\mathbf{0}\}}
\langle v,v\rangle$ the (squared Euclidean) norm of a shortest non-zero element
in a $d$-dimensional Euclidean lattice $\Lambda$. The \emph{determinant}
$\det(\Lambda)$ of $\Lambda$ is the determinant of a Gram matrix
(with coefficients $\langle b_i,b_j\rangle$ for a basis $b_1,\dots,b_d$
of $\Lambda=\sum_{i=1}^d\mathbb Z b_i$) for $\Lambda$. It equals the 
squared volume $\mathrm{Vol}(\mathbb R^d/\Lambda)^2$
of the flat $d$-dimensional torus $\mathbb R^d/\Lambda$. The 
Hermite invariant of $\Lambda$ is given by $\gamma(\Lambda)=
\frac{\min(\Lambda)}{\det(\Lambda)^{1/d}}$. It is invariant under rescalings 
and its $d/2$-th power is proportional to the packing-density of $\Lambda$.
Hermite's constant $\gamma_d=\max_{\Lambda,\dim(\Lambda)=d}\gamma(\Lambda)$
is the Hermite constant of a densest $d$-dimensional 
lattice-packing. Equivalently, $\gamma_d$ is equal to four times the squared
maximal injectivity radius of a suitable flat $d$-dimensional torus with
volume $1$.

We have now (see Proposition 2.1 of \cite{Ke}):

\begin{prop} We have 
$$I_d\leq \lfloor \gamma_d^{d/2}\rfloor\ .$$
\end{prop}

\proof{} We consider a sublattice 
$\Lambda'=\sum_{i=1}^d\mathbb Zb_i$ generated by 
$d$ linearly independent minimal vectors $b_1,\dots,b_d$ of minimum $1$ 
in a $d$-dimensional well-rounded lattice $\Lambda$ with minimum $1$.
We have thus $[\Lambda:\Lambda']^2\det(\Lambda)=\det(\Lambda')$.
Since $\Lambda'$ is generated by vectors of norm $1$ we have 
$\det(\Lambda')\leq 1$ showing $[\Lambda:\Lambda']^2\det(\Lambda)\leq 1$.
Since $\Lambda$ has minimum $1$ we have $\gamma(\Lambda)=\frac{1}{\det(\Lambda)^{1/d}}\leq \gamma_d$ implying $\det(\Lambda)\geq \gamma_d^{-d}$. We have thus
$[\Lambda:\Lambda']\leq \gamma_d^{d/2}$ and we get the inequality for $I_d$
by considering a suitable pair $\Lambda'\subset \Lambda$ achieving 
the natural integer $I_d$.\endproof

The upper bound $\gamma_d\leq (4/3)^{(d-1)/2}$ (see for example page 36 
of \cite{K}) gives the very bad upper bound
$I_d\leq (4/3)^{d(d-1)/4}$.

Blichfeldt's upper bound
\begin{align}\label{blichbdgamma}
\gamma_d&\leq \frac{2}{\pi}\Gamma(2+d/2)^{2/d}
\end{align} (see for example page 42 of \cite{K})
is more interesting and leads to 
\begin{align}\label{blichbdgamma}
I_d&\leq \left(\frac{2}{\pi}\right)^{d/2}\Gamma(2+d/2)
\end{align}
which is roughly $2^{1+d/2}/d$ times better than (\ref{ineqMink}). 
The proof of (\ref{ineqMink}) is however much more elementary
than the proof of (\ref{blichbdgamma}).

The first few values of the upper bounds (\ref{ineqMink}) and
(\ref{blichbdgamma}) are
$$\begin{array}{c||cccccccccc}
d&1&2&3&4&5&6&7&8&9&10\\
\hline
(\ref{ineqMink})&1&1&1&3&6&12&27&63&155&401\\
(\ref{blichbdgamma})&1&1&1&2&3&6&10&19&37&75\\
I_d&1&1&1&2&2&4&8&16&16&?
\end{array}\ .$$
The last row contains the correct
values $I_1,\dots,I_9$ which follow from Theorem 1.1 in \cite{Ke}.

\section{Enumerating pairs of lattices}

We denote by $\sigma_d(N)$ the
number of distinct lattices $\Lambda$ containing
$\mathbb Z^d$ as a sublattice of index $N$. Considering dual lattices,
we see that the number $\sigma_d(N)$ is also equal to the 
number of distinct sublattices of index $N$ in $\mathbb Z^d$.

\begin{lem} \label{lemsdn} We have 
$$\sigma_d(N)\leq N^d$$
for the number $\sigma_d(N)$ of lattices containing $\mathbb Z^d$ as
a sublattice of index $N$.
\end{lem}

\proof{} Every such lattice $\Lambda$ contains sublattices
$\Lambda_0=\mathbb Z^d\subset \Lambda_1\subset\dots\subset \Lambda_k=\Lambda$
with $\Lambda_i/\Lambda_{i-1}$ cyclic of prime order $p_i$
such that $p_1\leq p_2\leq \dots \leq p_k$ are all
prime-divisors of $N
=\prod_{i=1}^kp_i$, with multiplicities taken into account. 
Since $\mathbb Z^d$ is contained with prime index $p$ in exactly 
$(p^d-1)/(p-1)$ overlattices, we have
$\sigma_d(N)\leq \prod_{i=1}^k\frac{p_i^d-1}{p_i-1}\leq \prod_{i=1}^kp_i^d=N^d$.
\endproof

\begin{rem}
A nice exact formula for the number $\sigma_d(N)$ of subgroups of index $N$ in $\mathbb Z^d$ is given for example in \cite{Z} or \cite{G} and 
equals
\begin{align}\label{formexactsigma}
\sigma_d(N)&=\prod_{p\vert N}\left(\begin{array}{cc}e_p+d-1\\d-1
\end{array}\right)_p
\end{align}
where 
$\prod_{p\vert N}p^{e_p}=N$ is the factorization of $N$ into prime-powers
and where
$$\left(\begin{array}{cc}e_p+d-1\\d-1
\end{array}\right)_p=\prod_{j=1}^{d-1}\frac{p^{e_p+j}-1}{p^j-1}$$
is the evaluation of the $q$-binomial
$$\left[\begin{array}{cc}e_p+d-1\\d-1
\end{array}\right]_q=\frac{[e_p+d-1]_q!}{[e_p]_q!\ [d-1]_q!}$$
(with $[k]_q!=\prod_{j=1}^k\frac{q^j-1}{q-1}$) at the prime-divisor $p$
of $N$.

Lemma \ref{lemsdn} follows of course also from (\ref{formexactsigma}) 
and from the inequality $$\left(\begin{array}{cc}k+d-1\\d-1
\end{array}\right)_q\leq q^{kd}$$
which holds by induction on $k\in \mathbb N$ for $d\geq 1,\ q\geq 2$
(and which is asymptotically exact for $k=1$ since
$\left(\begin{array}{cc}d\\d-1
\end{array}\right)_2=\frac{2^d-1}{2-1}$).
\end{rem}

\begin{prop}\label{propmajperfslatt} 
A $d$-dimensional lattice $\Lambda$ is contained with index $\leq h$ in
at most $h^{d+1}$ different $d$-dimensional overlattices.
\end{prop}

\proof{} Lemma \ref{lemsdn} implies that the number
of such overlattices equals at most
$\sum_{n=1}^h\sigma_d(n)\leq \sum_{n=1}^hn^d\leq h\cdot h^d$. 
\endproof

\section{An upper bound for the number of perfect lattices}

\begin{thm}\label{thmbound} Up to similarities, there exist at most 
$$I_d^{d+1}{3^dI_d\choose {d\choose 2}}$$
$d$-dimensional perfect lattices
where $I_d$ is as in Section \ref{sectId}.
\end{thm}

Theorem \ref{thmbound} becomes effective after replacing $I_d$
for example with the upper bound (\ref{ineqMink}) or
(\ref{blichbdgamma}).

The main ingredient for proving Theorem \ref{thmbound} is
the following easy observation which is of independent interest:

\begin{lem}\label{lemfond}
Let $v_1,\dots,v_d\in\Lambda_{\min}$ be 
$d$ linearly independent minimal elements in
a $d$-dimensional perfect lattice $\Lambda$ such that the index $I$ of the
sublattice $\Lambda'=\sum_{i=1}^d \mathbb Zv_i$ in $\Lambda$ is maximal.

If $\mathcal C=
\{w\in\Lambda\otimes_{\mathbb Z}\mathbb R\ \vert\ 
w=\sum_{i=1}^dx_iv_i,\ x_i\in[-1,1]\}$ denotes 
the centrally symmetric $d$-dimensional parallelogram of all elements
having coordinates in $[-1,1]$ with respect to the basis $v_1,\dots,v_d$,
then all minimal vectors of $\Lambda$ belong to $\mathcal C$.
\end{lem}

Observe that the convex set $\mathcal C$ of Lemma \ref{lemfond} is simply
the $1$-ball of the $\parallel\quad \parallel_\infty$-norm
$\parallel \sum_{i=1}^d x_iv_i\parallel_\infty=\max_i\vert x_i\vert$ with
respect to the basis $v_1,\dots,v_d$.

\proof[Proof of of Lemma \ref{lemfond}]
Otherwise there exists $w\in\Lambda_{\min}$ such that
$w=\sum_{i=1}^d\beta_iv_i$ with $\vert\beta_j\vert>1$
for some index $j\in\{1,\dots,d\}$.
Exchanging $v_j$ with $w$ leads then to a sublattice (generated 
by the $d$ linearly independent minimal vectors 
$v_1,\dots,\hat{v_j},\dots,v_d,w$) of index strictly 
larger than $I$ in $\Lambda$.
\endproof

\proof[Proof of Theorem \ref{thmbound}]
Given a perfect $d$-dimensional lattice $\Lambda$, we choose
a set $v_1,\dots,v_d$ of $d$ linearly independent minimal elements
satisfying the condition of Lemma \ref{lemfond}. 
We denote by $I$ the index of the sublattice 
$\Lambda'=\sum_{i=1}^d \mathbb Zv_i$ 
of maximal index $I$ in $\Lambda$.
We consider the convex set
$\mathcal C=
\{w\in\Lambda\otimes_{\mathbb Z}\mathbb R\ \vert\ 
w=\sum_{i=1}^dx_iv_i,\ x_i\in[-1,1]\}$ containing 
$\Lambda_{\min}$ according to Lemma \ref{lemfond}.

We can now extend $v_1,\dots,v_d$ to a perfect 
subset $v_1,\dots,v_d,v_{d+1},\dots,v_{d+1\choose 2}$ of $\Lambda_{\min}\cap
\mathcal C$. Since every element of $\Lambda$ has at most
$3^d$ representatives modulo $\Lambda'$ in $\mathcal C$, the ${d\choose 2}$
elements $v_{d+1},\dots,v_{d+1\choose 2}$ belong to the finite subset 
$\mathcal C\cap \Lambda$
containing at most $3^dI$ elements. There are thus at most
${3^dI\choose{d\choose 2}}$ possibilities for the euclidean 
structure of $\Lambda$
(which is determined up to a scalar by the ${d+1\choose 2}$ minimal
vectors of the perfect set $v_1,\dots,v_{d+1\choose 2}$).

Proposition \ref{propmajperfslatt} gives the upper bound $I_d^{d+1}$
for the number of all overlattices $\Lambda$ containing
$\sum_{i=1}^d\mathbb Zv_i$ as a sublattice of index at most $I_d$.

This yields the upper bound $I_d^{d+1}{3^dI_d\choose {d\choose 2}}$ for
the number of different perfect $d$-dimensional lattices  (up to 
similarity).
\endproof

\begin{rem}\label{remimprbd} Minimal vectors of a perfect lattice $\Lambda$ do in 
general generate a perfect sublattice of $\Lambda$. 
This does not invalidate Theorem \ref{thmbound}.
The aim of the vectors $v_1,\dots,v_{d+1\choose 2}$
is only to pin down the Euclidean metric. The lattice $\Lambda$
is determined by one of the $I_d^{d+1}$ possible choices 
of a suitable overlattice 
of $\sum_{i=1}^d \mathbb Zv_i$.

The reader should also be aware that $\mathcal C$ does in general not contain
the Euclidean unit ball defined by the perfect set $v_1,\dots,v_{d+1\choose 2}$
of $\Lambda$. 

The bound of Theorem \ref{thmbound} can be improved to 
$$I_d^{d+1}{\lfloor(3^dI_d-(2d+1))/2\rfloor\choose {d\choose 2}}\ .$$
It is indeed enough to choose ${d\choose 2}$ suitable pairs of opposite
elements in $(\mathcal C\cap \Lambda)\setminus\{{\mathbf 0},\pm v_1,\dots,\pm v_d\}$ which contains at most $\lfloor(3^dI_d-(2d+1))/2\rfloor$ such pairs.
\end{rem}

\section{Digression: Symmetric lattice polytopes}\label{sectlattpol}

A lattice polytope is the convex hull of a finite number of lattice-points
in $\mathbb Z^d$. Its dimension is the dimension of its interior.
We call such a polytope $P$ symmetric if $P=-P$.
The group $\mathrm{GL}_d(\mathbb Z)$ acts on the set of $d$-dimensional 
symmetric lattice polytopes. A slight variation of the proof of Theorem 
\ref{thmbound} shows the following result:

\begin{thm}\label{thmuppbdpolyt} There are at most 
$$(d!)^{d+1}2^{3^d d!}$$
different $\mathrm{GL}_d(\mathbb Z)$-orbits of $d$-dimensional 
symmetric lattice polytopes containing no non-zero elements of $\mathbb Z^d$
in their interior.
\end{thm}

Theorem \ref{thmuppbdpolyt} gives also an upper bound on the number
of perfect lattices since convex hulls of minimal vectors of perfect 
lattices are obviously symmetric lattice polytopes containing no interior
non-zero lattice points (non-similar perfect $d$-dimensional 
lattices give obviously rise to polytopes in different 
$\mathrm{GL}_d(\mathbb Z)$-orbits). 
The bound of  Theorem \ref{thmuppbdpolyt} is
of course much worse than the bound given by Theorem \ref{thmbound}.

\subsection{The maximal index of hollow sublattices}

A sublattice $\Lambda'$ of a $d$-dimensional lattice
$\Lambda$ is called \emph{hollow} if $\Lambda'$ is generated
by $d$ linearly independent elements
$v_1,\dots,v_d$ of $\Lambda$ such that 
the interior of the convex hull spanned by $\pm v_1,\dots,\pm v_d$ contains 
no non-zero elements of $\Lambda$.

Hollowness is defined only 
in terms of convexity and is independent of metric properties.

We denote $H_d$ the maximal index of a hollow sublattice of 
a $d$-dimensional lattice. 

\begin{prop}\label{prophollow} We have 
$H_d\leq d!$.
\end{prop}

\proof[Proof of Proposition \ref{prophollow}] 
We can work without loss of generality with $\Lambda=\mathbb Z^d\subset
\mathbb R^d$. We denote by $\mathcal C$ the convex hull of generators
$\pm v_1,\dots,\pm v_d$ (satisfying the condition of hollowness) of 
a hollow sublattice $\Lambda'=\sum_{i=1}^d \mathbb Z v_i$ of $\mathbb Z^d$. The volume
$\mathrm{vol}(\mathcal C)$ of $\mathcal C$ equals $\frac{2^d}{d!}I$ where
$I$ is the index of $\Lambda'$ in $\mathbb Z^d$. Minkowski's inequality
$\mathrm{vol}(\mathcal C)\leq 2^d$ implies the result.
\endproof

\begin{rem}
Since sublattices generated by $d$ linearly independent minimal vectors
of a $d$-dimensional well-rounded lattice are always hollow, we 
have $I_d\leq H_d$ giving the bad upper bound $I_d\leq d!$
for the integers $I_d$ introduced in Section \ref{sectId}.
\end{rem}

\subsection{Proof of Theorem \ref{thmuppbdpolyt}}

\proof{} We consider such a lattice polytope $P$ with vertices in $\mathbb Z^d$.
We choose a set $v_1,\dots,v_d$ of $d$ linearly independent vertices of 
$P$ generating a sublattice $\Lambda'=\sum_{i=1}^d\mathbb Z v_i$
of maximal index $I$ in $\mathbb Z^d$. An obvious analogue of 
Lemma \ref{lemfond} holds and shows that all vertices of $P$ 
are contained in $\mathcal C=\{w\in\mathbb R^d\ \vert\ w=
\sum_{i=1}^d x_i v_i, x_i\in [-1,1]\}$. Since $\mathcal C$ intersects
$\mathbb Z^d$ in at most $3^dI$ elements there are at most 
$2^{3^dI}$ possibilities for choosing the vertices of $P$. 
Since $I\leq H_d\leq d!$ by  
Proposition \ref{prophollow} and since there at at most $H_d^{d+1}
\leq (d!)^{d+1}$ possibilities (see Lemma \ref{lemsdn})
for superlattices containing $\Lambda'$ with index at most $H_d\leq d!$,
we get the result.
\endproof

\begin{rem} The upper bound in Theorem \ref{thmuppbdpolyt} can easily 
be improved to $(d!)^{d+1}2^{\lfloor (3^d d!-(2d+1))/2\rfloor}$,
see Remark \ref{remimprbd}.
\end{rem}

\subsection{Lattice-polytopes defined by short vectors}

$d$-dimensional polytopes defined as convex hulls of minimal vectors in
Euclidean lattices define of course symmetric lattice polytopes 
with no non-zero lattice-points in their interior. The following result
is a slight generalization of this construction:

\begin{prop}\label{propconstrlattpol} (i) Let $\Lambda$ be a $d$-dimensional lattice of minimum $1$.
We have $\mathring{P}\cap \Lambda=\{{\mathbf 0}\}$ if $\mathring{P}$ is
the interior of a $d$-dimensional polytope $P$ defined as the convex hull of a 
set $\mathcal V=-\mathcal V
\subset \Lambda\setminus\{{\mathbf 0}\}$ of pairs of opposite non-zero 
elements of squared Euclidean
length $\leq 2$ in $\Lambda$.

(ii) We have moreover $P\cap \Lambda=\mathcal V\cup \{{\mathbf 0}\}$
and $\mathcal V$ is the set of vertices of $P$ if the inequality is
strict (ie. if all elements of $\mathcal V$ are of squared Euclidean 
length strictly smaller than $2$).
\end{prop}

Taking $\mathcal V=\{\pm b_i,\pm b_i\pm b_j\}$ for $b_1,\dots,b_d$ 
an orthogonal basis of the standard lattice $\mathbb Z^d$ shows 
that the bound $2$ is sharp in part (ii).

\begin{proof} Suppose that $u\not\in \mathcal V$
belongs to $P\cap \Lambda$.
There exists thus $d$ elements $v_1,\dots,v_d\in \mathcal V$ such
that 
$u$ is contained in the simplex $\Sigma$ spanned by the origin 
${\mathbf 0}$ and by $v_1,\dots,v_d$.
Since $\Lambda$ has minimum $1$, the element $u$ 
belongs to the subset $\sigma\subset \Sigma$ of all elements 
of $\Sigma$ which are at distance $\geq 1$ from $v_1,\dots,v_d$.
The norm of $u$ is thus at most equal to an element $w$ of maximal 
norm in $\sigma$. For such an element $w$ there exists an index $i$ 
such that $w$ is at distance exactly $1$ from $v_i$ and such that 
$\langle w,v_i-w\rangle\geq 0$. This implies that $w$ is of norm at most
$1$ with equality if and only if $w$ is on the boundary $\partial P$ of $P$.
\end{proof}

\begin{rem} The constant $2$ in part (i) of Proposition \ref{propconstrlattpol}
is perhaps not optimal. The sequence $c_1=4,c_2,\dots$ 
of upper bounds $c_d$ on the set of 
all possible constants for assertion (i)
of Proposition \ref{propconstrlattpol})
is decreasing to a limit $\geq 2$.
\end{rem}

\section{Proof of Theorem \ref{thmmain}}

\proof{}
Proposition \ref{propauxexp}
counts the number $\alpha(d-8)$ of $d$-dimensional 
lattices $L_d(h_1,\dots)$ satisfying the conditions 
of Theorems \ref{thmperf}  and \ref{thmautom}.
All these lattices are perfect by Theorem \ref{thmperf} and 
non-isomorphic (and thus also non-similar
since they have the same minimum) by Theorem \ref{thmautom}.
Their number $\alpha(d-8)$ grows exponentially fast by 
Proposition \ref{propalphan}. This implies eventually $p_d>e^{e^{d-\epsilon}}$.

The eventual upper bound on $p_d$ follows from Theorem 
\ref{thmbound}.
\endproof

\section{Upper bounds for cells}

Similarity classes of
$d$-dimension perfect lattices correspond to cells of maximal dimension of
the Vorono\"{\i} complex, 
a finite ${d+1\choose 2}$-dimensional cellular complex encoding
information on $\mathrm{GL}_d(\mathbb Z)$. 

The proof of Theorem \ref{thmbound} can easily be modified to show
that this complex has at most $$I_d^{d+1}{3^dI_d\choose k+1}$$ 
(with $I_d$ replaced by the upper bound (\ref{ineqMink}) or
(\ref{blichbdgamma})) different cells
of dimension $k$. (Cells are in general defined by affine subspaces. 
This explains the necessity of choosing $k+1$ elements.) Slight 
improvements of this result are possible by choosing carefully the 
initial $d$ linearly independent minimal elements $v_1,\dots,v_d$.

Without error on my behalf,
these bounds are somewhat smaller than the bounds given in Proposition 2
of \cite{S}.

\section{Algorithmic aspects}

The proof of Theorem \ref{thmbound} can be made into a naive algorithm
as follows:

For a fixed dimension $d$, consider the list $\mathcal L$
of all overlattices $\Lambda$ of $\mathbb Z^d$ such that $\Lambda/\mathbb Z^d$
has at most $I_d$ elements.

Given a lattice $\Lambda\in\mathcal L$, 
construct the finite set $\mathcal C=\Lambda\cap [-1,1]^d$.
For every subset $\mathcal S$ of ${d\choose 2}$ elements in $\mathcal C$,
check if $\mathcal S\cup\{b_1,\dots,b_d\}$ (with $b_1,\dots,b_d$
denoting the standard basis of $\mathbb Z^d$)
is perfect. If this is the case, check if the quadratic form
$q$ defined by $q(v_i)=1$ for $i=1,\dots,{d+1\choose 2}$ is positive definite. 

Check finally that the lattice $\Lambda$ (endowed with this quadratic
form) has no non-zero elements of length shorter than $1$ and 
add the resulting perfect lattice to your list $\mathcal P$ of
perfect $d$-dimensional lattices if this is the case and if 
$\mathcal P$ does lack it (up to similarity).

This algorithm can be accelerated using the following facts:

\begin{enumerate}
\item{} It is enough to consider overlattices $\Lambda$ containing 
$b_1,\dots,b_d$ as a hollow set.
We can reduce the list $\mathcal L$ of 
overlattices by taking only one overlattice in each orbit 
of the group $\mathcal S_d\ltimes \{\pm 1\}^d$ acting linearly 
(by coordinate permutations and sign changes) on the set
$\pm b_1,\dots,\pm b_d$.

\item{} The list $\mathcal C$ can be made smaller: First of all,
its size can be divided by $2$ by considering only elements 
with strictly positive first non-zero coordinate.
Secondly, an element $v$ can only lead to a perfect set
if the convex hull of $\pm b_1,\dots,\pm b_d,\pm v$
intersects $\Lambda\setminus\{0\}$ only in its vertices.
This idea can be refined since the same property holds
for the convex hull of $\pm b_1,\dots,b_d,\pm v_{d+1},\dots,
\pm v_{d+k}$ for the subset $b_1,\dots,b_d,v_{d+1},\dots,v_{d+k}$
of the first $d+k$ elements of the perfect set 
$\{b_1,\dots,b_d\}\cup\{v_{d+1},\dots,v_{d+1\choose 2}\}$.

\item{} Writing $v_i=\sum_{j=1}^i\alpha_{i,j}b_j$
we get a ${d+1\choose 2}\times d$ matrix whose minors
are all in $[-1,1]$ (otherwise we get a contradiction 
with maximality of the index $I$ of the sublattice
$\mathbb Z^d$ generated by 
$v_1=b_1,\dots,v_d=b_d$).
\end{enumerate} 

\section{Improving $I_d$?}

The upper bounds (\ref{ineqMink}) and
(\ref{blichbdgamma}) for $I_d$ are not tight.
They can however not be improved too much as shown below.

First of all, let us observe that the maximal index $I$
(of sublattices generated by $d$ linearly independent minimal vectors of 
a perfect lattice) can be as small as $1$:

\begin{prop}\label{propAd} Every set of $d$ linearly independent roots of the root lattice
$A_d$ generates $A_d$.
\end{prop}

I ignore if there exist other perfect lattices generated by any maximal
set of linearly independent minimal vectors.

\proof[Proof of Proposition \ref{propAd}] We consider the usual realisation of $A_d$ as the sublattice
of all elements in $\mathbb Z^{d+1}$ with coordinate sum $0$.
Roots are given by $b_i-b_j$ where $b_1,\dots,b_{d+1}$ is the 
standard orthogonal basis of $\mathbb Z^{d+1}$.
Let $\mathcal B=\{f_1,\dots,f_d\}$ be a set of $d$ linearly independent roots of $A_d$.
We associate to $\mathcal B$ a graph with vertices $1,\dots,d+1$.
An element $b_i-b_j\in\mathcal B$ yields an edge joining the vertices $i,j$.
The resulting graph has $d$ edges and it has to be connected (otherwise
we get a sublattice of rank $d+1-k$ in $\mathbb Z^d$ where $k$ denotes
the number of connected components). It is thus a tree
$T_{\mathcal B}$. Roots of $A_d$
correspond to the ${d+1\choose 2}$ shortest paths joining two distinct vertices
of  $T_{\mathcal B}$. It follows that $\mathcal B$
generates the root lattice $A_d$.
\endproof

The corresponding situation is fairly different for root lattice $D_{2d}$
of even dimension:

\begin{prop}\label{propDn} The maximal index of a sublattice generated by $2d$
linearly independent roots of $D_{2d}$ equals $2^{d-1}$.
\end{prop}

\proof[Proof of Proposition \ref{propDn}] We consider 
$f_1=b_1+b_2,f_2=b_1-b_2,f_3=b_3+b_4,f_4=b_3-b_4,\dots,f_{2d-1}=b_{2d-1}+b_{2d},
f_{2d}=b_{2d-1}-b_{2d}$. These elements generate a sublattice of index $2^{d-1}$.
Since they are orthogonal, this index is maximal.
\endproof

Proposition \ref{propDn} dashes hope for big improvements on the upper
bound $I_d$: Any such improvement has necessarily exponential growth. 
We have however no candidate (of a sublattice generated by $d$
linearly minimal elements of a $d$-dimensional perfect lattice) 
such that the quotient group has elements of very large order.

\begin{rem} An analogue of Proposition \ref{propDn} holds also for
root lattices $D_{2d+1}$ but the proof is slightly more involved.
Indeed, such a lattice contains a sublattice generated by roots
of index $2^{d-1}$
corresponding to the root system $A_1^{2d-2}A_3$. For proving that
the index $2^{d-1}$ is maximal, show
first that $D_n$ cannot contain a root system of type $E$. 
It follows that $2d+1$ linearly independent roots of $D_{2d+1}$
generate a root system with connected components of type $A$ and $D$
whose ranks sum up to $2d+1$. Each such root system is possible
and easy computations show that $A_1^{2d-2}A_3$ maximises the index
among all possibilities.
\end{rem}

\section{Perfect lattices of small dimensions}

Since $I_2=I_3=1$, the classification of perfect lattices
of dimension $2$ and $3$ can easily be done by hand.

\subsection{Dimension $2$}

There is essentially (i.e. up to action of the dihedral group of isometries
of the square with vertices $\pm b_1,\pm b_2$) only one way to extend
$v_1=b_1=(1,0),v_2=b_2=(0,1)$ to a perfect set. It is given by 
choosing $v_3=b_1+b_2=(1,1)$ and leads to the root lattice $A_2$.

\subsection{Dimension $3$}

There are only three essentially different ways (up to the 
obvious action of the group $\mathcal S_3\ltimes\{\pm 1\}^3$) 
for enlarging
$v_1=b_1,v_2=b_2,v_3=b_3$ to a perfect set in $\{0,\pm 1\}^3$.

The first one, given by $v_4=b_1+b_2+b_3=(1,1,1),v_5=b_1+b_2=(1,1,0),
v_6=b_2+b_3=(0,1,1)$ 
(with Gram matrix $\frac{1}{2}\left(\begin{array}{ccc}2&-1&0\\-1&2&-1\\0&-1&2
\end{array}\right)$) leads to the root lattice $A_3$ in its standard 
realisation.

The second one, given by $v_4=b_1+b_2,v_5=b_2+b_3,v_6=b_1+b_3$
leads to the degenerate quadratic form 
$\frac{1}{2}\left(\begin{array}{ccc}2&-1&-1\\-1&2&-1\\-1&-1&2
\end{array}\right)$.

The third one, $v_4=b_1+b_2,v_5=b_2+b_3,v_6=b_1-b_3$
defines also $A_3$. It
leads to the quadratic form 
$\frac{1}{2}\left(\begin{array}{ccc}2&-1&1\\-1&2&-1\\1&-1&2
\end{array}\right)$ and corresponds to 
$$\begin{array}{crrrr}
v_1=&(1&-1&0&0)\\
v_2=&(0&1&-1&0)\\
v_3=&(0&-1&0&1)\\
v_4=&(1&0&-1&0)\\
v_5=&(0&0&-1&1)\\
v_6=&(1&0&0&-1)\\
\end{array}$$
with respect to the standard realisation 
$$A_3=\{(x_1,x_2,x_3,x_4)\in \mathbb Z^4\ \vert\ x_1+x_2+x_3+x_4=0\}$$
of the root lattice $A_3$.

\subsection{Dimension $4$}

We give no complete classification in dimension $4$ but we describe 
briefly how the two $4$-dimensional perfect lattices fit into our
framework.

In dimension $4$ we have to consider overlattices containing $\mathbb Z^4$
with index $1$ or $2$. The known classification shows that 
the root lattices $A_4$ and $D_4$ are (up to
similarities) the only perfect lattices in dimension $4$. Every 
linearly independent set of $4$ roots of $A_4$ generates $A_4$ by 
Proposition \ref{propAd}. The root lattice $A_4$ is thus obtained by 
considering (for example) the $10={5\choose 2}$ vectors
of $\{0,1\}^4$ with consecutive coefficients $1$. The first $4$ vectors 
can be chosen as the standard basis
$v_1=(1,0,0,0),v_2=(0,1,0,0),v_3=(0,0,1,0),v_4=(0,0,0,1)$.
The remaining six vectors are 
$$(1,1,0,0),(1,1,1,0),(1,1,1,1),(0,1,1,0),(0,1,1,1),(0,0,1,1)\ .$$

Proposition \ref{propDn} shows that four linearly independent roots of 
$D_4$ generate a sublattice of index at most $2$ in $D_4$. The index is 
exactly $2$ if and only if the four roots are pairwise orthogonal. 
Given four such orthogonal roots $v_1,\dots,v_4$, the remaining eight pairs of 
roots are given by $\pm \frac{1}{2}(v_1\pm v_2\pm v_3\pm v_4)$.

\section{Bases of small height for perfect lattices}

\begin{thm}\label{thmsmallheight} A $d$-dimensional perfect lattice has a $\mathbb Z$-basis
$f_1,\dots,f_d$ such that the coordinates $\alpha_i\in \mathbb Z$
(with respect to the basis $f_1,\dots,f_d$)
of a minimal element $v=\sum_{i=1}^d\beta_if_i$
satisfy the inequalities $\vert \beta_i\vert\leq 2^{i-1}I_d,\ i=1,\dots,d$
with $I_d$ as in Section \ref{sectId}.
\end{thm}

The main ingredient of the proof is the following result:

\begin{lem}\label{lemsmallheight} Let $v_1,\dots,v_d$ be $d$ linearly
independent elements generating a sublattice $\sum_{i=1}^d\mathbb Zv_i$
of index $I$ in 
a $d$-dimensional lattice $\Lambda$. There exist a $\mathbb Z$-basis
$f_1,\dots,f_d$ of $\Lambda=\sum_{i=1}^d\mathbb Z f_i$ such that we have 
$v_i=\sum_{j=i}^d\alpha_{i,j}f_j$ and $\vert\alpha_{i,j}\vert\leq
2^{\max(0,j-i-1)}I$ for all $i$ and all $j\geq i$.
\end{lem}

The proof of Lemma \ref{lemsmallheight} shows in fact slightly more
since it constructs such a basis $f_1,\dots,f_d$ of $\Lambda$ 
in the fundamental domain $\sum_{i=1}^d[0,1]v_i$ of the 
lattice $\sum_{i=1}^d\mathbb Zv_i$.

\proof[Proof of Theorem \ref{thmsmallheight}]
We choose
minimal elements $v_1,\dots,v_d$ generating a sublattice 
$\sum_{i=1}^d \mathbb Zv_i$ of maximal index $I\leq I_d$ in a perfect
$d$-dimensional lattice $\Lambda$. Lemma \ref{lemfond}
shows that we have $v=\sum_{i=1}^d\lambda_iv_i$ with $\lambda_i\in[-1,1]$ 
for every minimal vector $v$ of $\Lambda$. With respect to a basis 
$f_1,\dots,f_d$ of $\Lambda$ as in Lemma \ref{lemsmallheight} we get
$$v=\sum_{i=1}^d\lambda_iv_i=\sum_{i=1}^d\lambda_i\sum_{j=i}^d\alpha_{i,j}f_j.$$
We have thus $v=\sum_{j=1}^d\beta_jf_j$ with 
$\beta_j=\sum_{i=1}^j\lambda_i\alpha_{i,j}$.
Since $\lambda_i\in[-1,1]$ and since $\vert\alpha_{i,j}\vert\leq 2^{\max(0,j-i-1)}I$
by Lemma \ref{lemsmallheight}, we get $\beta_j\leq 
I\sum_{i=1}^j2^{\max(0,j-i-1)}=2^{j-1}I$. We apply now the inequality $I\leq I_d$ of 
Proposition \ref{propIdbound}.
\endproof

\proof[Proof of Lemma \ref{lemsmallheight}]
The result clearly holds for $d=1$. 

Consider $d+1$ linearly independent elements $v_0,\dots,v_d$ generating a 
sublattice of index $I$ in a $(d+1)$-dimensional lattice $\Lambda$.
We denote by $\pi$ the linear form defined by $\pi(v_0)=1$ and 
$\pi(v_1)=\dots=\pi(v_d)=0$. The set $\pi(\Lambda)$ is of the form
$\frac{1}{a}\mathbb Z$ with $a$ dividing $I$ such that 
$\Lambda'=\ker(\pi)\cap \Lambda$ contains $\sum_{i=1}^d\mathbb Zv_i$ 
as a sublattice of index $I/a$. By induction, there exists 
a basis $f_1,\dots,f_d$
of $\Lambda'$ such that every element 
$v_1,\dots,v_i=\sum_{j=i}^d\alpha_{i,j}f_j,\dots,
v_d$ involves only coordinates $\alpha_{i,j}$ of absolute value at most
$2^{\max(0,j-i-1)}I/a$.
We choose now an element $f_0\in \Lambda$ with
$\pi(f_0)=\frac{1}{a}$ in the fundamental 
domain $\sum_{i=0}^d[0,1]v_i$ of $\mathbb R^{d+1}/\sum_{i=0}^d\mathbb Zv_i$.
We have thus $f_0=\sum_{i=0}^d\lambda_iv_i$ with 
$\lambda_0=\frac{1}{a}$ and $\lambda_1,\dots,\lambda_d\in[0,1]$.
We get thus
\begin{align*}
af_0&=v_0+a\sum_{i=1}^d\lambda_iv_i\\
&=v_0+a\sum_{i=1}^d\sum_{j=i}^d\lambda_i\alpha_{i,j}f_j.
\end{align*}
This shows that $v_0=\sum_{j=0}^d\beta_jf_j$ has coordinates 
$\beta_j$ given by $\beta_0=a$ and by
$$\beta_j=-a\sum_{i=1}^j\lambda_i\alpha_{i,j}$$
for $j=1,\dots,d$.
We have obviously $\vert\beta_0\vert=\vert a\vert\leq I$.
Since $\lambda_i\in[0,1]$ and since $\vert\alpha_{i,j}\vert\leq 2^{\max(0,j-i-1)}I/a$
we get $\vert\beta_j\vert\leq I\sum_{i=1}^j2^{\max(0,j-i-1)}=2^{\max(0,j-1)}I=2^{\max(0,j-0-1)}I$ for $j=1,\dots,d$.
\endproof

\section{Heuristic arguments for an improved upper bound}\label{sectionheuristics}

We present a few non-rigorous thoughts suggesting an eventual
upper bound of $e^{d^{2+\epsilon}}$ (for arbitrarily small
strictly positive $\epsilon$) for the numbers $p(d)$ of
$d$-dimensional perfect lattices, up to similarities.

Consider an increasing function $\alpha:\mathbb N\rightarrow \mathbb N$
such that $\lim_{d\rightarrow \infty}\frac{\log\log\alpha(d)}{\log d}=0$
(e.g. $\alpha(d)=\lceil d^{k(1+\log(d))}\rceil$ for some positive constant $k$).
We denote by $p_\alpha(d)$ the set of all similarity classes of 
$d$-dimensional integral perfect lattices
having a basis involving only elements of (squared euclidean) norm 
at most $\alpha(d)$. For any $\epsilon>0$, we have 
clearly $p_\alpha(d)<e^{d^{2+\epsilon}}$ for almost all $d\in\mathbb N$
since such lattices have Gram matrices in the set of all
$(1+2\alpha(d))^{d+1\choose 2}$ symmetric matrices with coefficients in 
$\{-\alpha(d),\dots,\alpha(d)\}$.
In order to have $p(d)<e^{d^{2+\epsilon}}$ for any $\epsilon>0$ and for 
any $d>N(\epsilon)$, it is now sufficient to have
\begin{align}\label{limpap}
\lim_{d\rightarrow\infty}\frac{p_\alpha(d)}{p(d)}e^{d^{2+\epsilon}}&=
\infty
\end{align}
for all $\epsilon>0$. 

In other terms, $p(d)=o\left(e^{d^{2+\epsilon}}\right)$
would imply that the proportion of similarity classes of
integral perfect lattices 
of dimension $d$ generated by, say, vectors shorter than $d^{100\log d}$
(with respect to all $p(d)$ perfect lattices) decays extremely fast.

Equivalently, we can consider the set $\tilde p_\alpha(d)$ 
of all similarity classes of 
$d$-dimensional integral perfect lattices
having $d$ linearly independent minimal elements generating a 
sublattice of index at most $\alpha$. The proof of Theorem
\ref{thmbound} shows that (\ref{limpap}) with $\tilde p_\alpha (d)$
replacing $p_\alpha(d)$ implies also the eventual inequalities
$p_d<e^{d^{2+\epsilon}}$.

{\bf Acknowledgements.} I thank A. Ash, E. Bayer, M. Decauwert, P. Elbaz-Vincent, L. Fukshansky and J. Martinet for interesting discussions,
remarks and comments.

\noindent Roland BACHER, 
 Univ. Grenoble Alpes, Institut Fourier, 
 F-38000 Grenoble, France.

\noindent e-mail: Roland.Bacher@univ-grenoble-alpes.fr

\end{document}